\documentclass[]{article}
\usepackage{amscd}
\usepackage{amsmath}
\usepackage{latexsym}
\usepackage{amsfonts}
\usepackage{amssymb}
\usepackage{amsthm}
\usepackage{bm}
\usepackage{graphicx,cases}
\usepackage{color,cite}
\usepackage{geometry}
\newtheorem{prop}{Proposition}[section]

\newtheorem{lemma}{Lemma}[section]
\newtheorem{theorem}{Theorem}[section]

\def\ra{\rangle}
\def\la{\langle}
\newcommand{\beq}{\begin{equation}}
\newcommand{\eeq}{\end{equation}}
\newcommand{\ben}{\begin{eqnarray}}
\newcommand{\een}{\end{eqnarray}}
\newcommand{\beno}{\begin{eqnarray*}}
\newcommand{\eeno}{\end{eqnarray*}}

\theoremstyle{remark}% \newtheorem{Cor}{Corollary}
\newtheorem{remark}{Remark}[section]

\begin{document}
%\begin{CJK*}{GBK}{song}
\title{Stability threshold for 2D shear flows near Couette \\of the Navier-Stokes equation }% without magnetic diffusion and heat convection}

\author{ {Dongfen Bian}  \\[1ex]
\normalsize School of Mathematics and Statistics,\\
\normalsize Beijing Institute of Technology, Beijing 100081, China \\ % Your institution
\normalsize {biandongfen@bit.edu.cn}
\and
{Xueke Pu%\footnote{Corresponding author: puxueke@gmail.com}
} \\[1ex] % Your name
\normalsize School of Mathematics and Information Science,\\
\normalsize Guangzhou University, Guangzhou 510006, China\\ % Your institution
\normalsize {xuekepu@gzhu.edu.cn} \\
}

\date{}

\maketitle
\begin{abstract}
In this paper, we consider the stability threshold of the 2D shear flow $(U(y),0)^{\top}$ of the Navier-Stokes equation at high Reynolds number $Re$. When the shear flow is near in Sobolev norm to the Couette flow $(y,0)^{\top}$ in some sense, we prove that if the initial data $u_0$ satisfies $\|u_0-(U(y),0)^{\top}\|\leq \epsilon Re^{-1/3}$, then the solution of the 2D Navier-Stokes equation approaches to some shear flow which is also close to the Couette flow for $t\gg Re^{1/3}$, as $t\to\infty$.

\end{abstract}
\begin{center}
 \begin{minipage}{120mm}
\end{minipage}
\end{center}
\begin{center}
 \begin{minipage}{120mm}
\end{minipage}
\end{center}

\section{Introduction}
\setcounter{section}{1}\setcounter{equation}{0}
In this paper, we consider the following 2D incompressible Navier-Stokes equation in $\Bbb T\times \bf R$:
\begin{equation}
\begin{cases}
\partial_tu+u\cdot\nabla u+\nabla P=\nu\Delta u,\\
\nabla\cdot u=0,\\
u|_{t=0}=u_{in}(x,y),
\end{cases}
\end{equation}
where $\nu>0$ denotes the viscosity, i.e.,  the inverse of the Reynolds number $Re$. When $u_{in}(x,y)=(U(y),0)^{\top}$, then the solution of the Navier-Stokes equation can be explicitly given by the following
\begin{equation}
\begin{split}
  u_S(t,x,y)=\left(
  \begin{array}{c}
       \bar U(t,y) \\
        0 \\
  \end{array}
  \right)=
  \left(
  \begin{array}{c}
       e^{\nu t\partial_y^2}U(y) \\
        0 \\
  \end{array}
  \right),
\end{split}
\end{equation} 
where the subscript $S$ stands for shear.

To study the stability of this shear flow, it is natural to consider the perturbation $v=u-u_S$, which, in the $(t,x,y)$ coordinate  system, satisfies
\begin{equation}
\begin{cases}
v_t+\bar U\partial_xv+v\cdot\nabla v-\nu\Delta v+\left(
      \begin{array}{c}
       \partial_y\bar Uv^2 \\
        0 \\
      \end{array}
 \right)
+\nabla p=0,\\
\nabla\cdot v=0,\ \ \ v|_{t=0}=v_{in}=u_{in}-(U(y),0)^{\top}.
\end{cases}
\end{equation}
In this paper, we will only consider the case when $(U(y),0)^{\top}$ is near the standard Couette flow $(y,0)^{\top}$ in the sense $\|\partial_yU-1\|_{H^{\sigma}}\ll1$ and $\|\partial_{yy}U(y)\|_{H^{\sigma}}\ll1$ for some large enough integer $\sigma>0$.

In terms of vorticity $\omega=\nabla^{\bot}\cdot v$, we have 
\begin{equation}\label{vorticity}
\begin{cases}
\omega_t+\bar U\partial_x\omega+v\cdot\nabla \omega=\bar U''\partial_x\varphi+\nu\Delta\omega,\\
\omega|_{t=0}(x,y)=\omega_{in}(x,y),\\
\Delta\varphi=\omega,\\
v=\nabla^{\bot}\varphi=\left(
      \begin{array}{c}
        -\partial_y\varphi\\
        \partial_x\varphi \\
      \end{array}
 \right).
\end{cases}
\end{equation}

When $U(y)=y$, the shear flow $(U(y),0)^{\top}$ reduces to the Couette flow $(y,0)^{\top}$, whose stability analysis dates back to the classical results of Rayleigh \cite{Rayleigh1880} and Kelvin \cite{Kelvin1887} for the incompressible fluid. The Couette flow is mathematically spectrally stable for all Reynolds numbers, although instability of Couette flow and transition to turbulence is observed experimentally for sufficiently high Reynolds number. This paradox gained much attention in the fluid mechanics. Another interesting phenomenon is the inviscid damping, as firstly observed by Orr \cite{Orr} that the velocity will tend to zero as $t\to\infty$, even for Euler equations. The nonlinear stability and inviscid damping around the Couette flow are sensitive to the topology of the perturbation. Indeed, Lin and Zeng \cite{LZ11} showed that nonlinear inviscid damping is not true in any vorticity $H^s\ (s<3/2)$ neighborhood of Couette flow for any horizontal period, while Bedrossian and Masmoudi \cite{BM15} proved nonlinear inviscid damping around the Couette flow in Gevrey class $2_{-}$.

The purpose of this paper is to study the long-time dynamics of the perturbatio $\omega$ solving \eqref{vorticity} in the high Reyolds number limit $\nu\to0$. Compared to (1.3) in \cite{MZ19}, we have an extra source term $\bar U''\partial_x\varphi$ in the vorticity equation. This leads to the following linearization equation that is different to the linearization studied by Rayleigh \cite{Rayleigh1880} and Keilvin \cite{Kelvin1887},
\begin{equation}
\begin{cases}
\omega_t+\bar U\partial_x\omega-\bar U''\partial_x\varphi-\nu\Delta\omega=0,\\
\omega|_{t=0}(x,y)=\omega_{in}(x,y),\\
\Delta\varphi=\omega.
\end{cases}
\end{equation}

To study the stability of the shear flow, the goal is, given a norm $\|\cdot\|_{X}$, to determine an exponent $\gamma=\gamma(X)$ such that $\|u_{in}\|_X\lesssim \nu^{\gamma}$ implies stability and $\|u_{in}\|_X\gtrsim  \nu^{\gamma}$ implies possible instability. In the  applied  literature, $\gamma$ is sometimes referred to as the  transition threshold. In the case of the 3D Couette flow, it was shown that $\gamma=1$ when $X$ is taken as Gevrey-$m$ with $m<2$ in \cite{BGM15a,BGM15b} and $\gamma\leq3/2$ when $X$ is taken as Sobolev space $H^{s}$ for $s>7/2$ in \cite{BGM17}. In the case of the 2D Couette flow, it was shown that $\gamma=0$ when $X$
 is taken as Gevrey-$m$ with $m<2$ in \cite{BMV16}. It was also shown that $\gamma=1/2$ for 2D shear flow near Couette flow in Sobolev spaces \cite{BVW18}. This threshold was improved to $\gamma=1/3$ for the 2D shear flow in sufficiently regular spaces $H^{\sigma}$ for $\sigma\geq40$ by Masmoudi and Zhao \cite{MZ19}.

In the paper of Bedrossian, Vicol and Wang \cite{BVW18}, the authors showed that for a shear flow $(U(y),0)^{\top}$ close to the Couette flow $(y,0)^{\top}$ with $\|U(y)-y\|_{H^{N+4}}\ll 1$, the Sobolev stability threshold is $\gamma=1/2$. That is to say, if the initial perturbation is $\epsilon\ll \nu^{1/2}$ close to the shear flow, then the solution of the 2D Navier-Stokes equation remains $\epsilon$ close to $(e^{t\nu\partial_{yy}}U(y),0)^{\top}$ for all $t>0$. Our main goal is to improve this threshold to $\gamma=1/3$. We aimed to show that when $\|u_{in}-U(y)\|_{H^{\sigma}}=\epsilon \lesssim \nu^{1/3}$, then the solution is stable and $\epsilon$ close to the shear flow for $t>0$. We also want to stress that to obtain the improved stability threshold $\gamma=1/3$ of a general shear flow $(U(y),0)^{\top}$, we would assume that that the shear flow is very close to the Couette flow $(y,0)^{\top}$. In particular, we would assume  $\|U(y)-y\|_{H^{\sigma}}\leq \epsilon\nu^{1/3}$ for some small constant $\epsilon$ independent of $\nu$. To estimate the source term caused by the term $\bar U''\partial_x\varphi$, since generally $\bar U''\neq 0$ in contrast to the Couette flow case, we also require a decay of the background solution, which is insured by assuming $\|U(y)-y\|_{L^1}\leq \epsilon\nu^{5/4}$ with $\epsilon\leq \epsilon_0$ for some small positive $\epsilon_0$ and utilizing the decay of the one dimensional heat equation.

\subsection{Main results}
Below, we state the main result in this paper.
\begin{theorem}\label{thm1.1}
For $\sigma\geq40$, $\nu>0$, there exist $0<\epsilon_0,\nu_0<1$ such that for all $0<\nu<\nu_0$ and $0<\epsilon<\epsilon_0$, if the initial shear $U$ satisfies % $\|U\|_{H^{\sigma}}\leq \epsilon\nu^{1/3}(???)$, 
$\|U(y)-y\|_{H^{\sigma}}\leq \epsilon\nu^{1/3}$, $\|U(y)-y\|_{L^1}\leq \epsilon\nu^{5/4}$  and the initial perturbation $\omega_{in}$ satisfies $\|\omega_{in}\|_{H^{\sigma}}\leq \epsilon\nu^{1/3}$, then the solution $\omega(t)$ satisfies the following properties:\\
(1) global stability in $H^{\sigma}$,
\begin{equation}
\|\omega(t,x+t\bar U(t,y)+\Phi(t,y),y)\|_{H^{\sigma}}\leq C\epsilon\nu^{1/3},
\end{equation}
where $\Phi(t,y)$ is given explicitly by
$$\Phi(t,y)=\int_0^te^{\nu(t-\tau)\partial^2_y} \left(\frac{1}{2\pi}\int_{\Bbb T}v^x(\tau,x,y)dx\right)d\tau.$$
\\
(2) Inviscid damping
$$\|P_{\neq}v^x\|_2+\la t\ra\|v^y\|_2\lesssim \frac{C\epsilon \nu^{1/3}}{\la t\ra\la \nu t^3\ra}.$$\\
(3) Weak enhanced dissipation,
$$\|P_{\neq}\omega(t)\|_2\lesssim  \frac{C\epsilon \nu^{1/3}}{ \la \nu t^3\ra}.$$
The constant $C$ is independent of $\nu$ and $\epsilon$.
\end{theorem}

\begin{remark}
Stronger enhanced dissipation can be obtained in the form $\|\Omega_{\neq}\|_L^2\leq C\epsilon\nu^{1/3}\la \nu t^3\ra^{-\alpha}$ for $\alpha\geq1$, as in \cite{MZ19}. We also want to remark that the proof of this theorem depends heavily on the proof of the paper of \cite{MZ19}, where the Masmoudi and Zhao obtained $\gamma=1/3$ stability threshold of the 2D Couette flow in Sobolev spaces $H^{\sigma}$ for $\sigma\geq40$. Because of this, many details of the estimates, in particular those that have appeared in their paper, are omitted and only sketched in the present paper. The main multiplier $A^{\sigma}_k$, the multiplier $A^s_E$ quantifying the enhanced dissipation both have been introduced in \cite{MZ19} and hence the details of the construction are omitted for conciseness.
\end{remark}

\begin{remark}
	After this paper is complete, we found a paper of Li, Masmoudi and Zhao [arXiv:2203.10894v1] submitted to arXiv.org on 21 March.  Their paper also deals with the stability of the shear flows near Couette flow of the Navier-Stokes equation or the Euler equation. However, their results are different from ours in several aspects. We refer the interested readers to their paper \cite{LMZ22}.  
\end{remark}

The main contribution is two folds. The background shear flow in this paper is quite general, although near the Couette in the Sobolev spaces. This yields an extra term $\bar U''\partial_x\varphi$ in the vorticity equation. In the standard Couette flow case, the vorticity equation lacks this linear source term and the Biot-Savart law is a Fourier multiplier in the new variables. Secondly and more importantly, the main contribution is a new change of coordinate to serve our purpose that we are aimed to obtain an improved stability  threshold $\gamma=1/3$ for a general shear flow. Indeed, the change of coordinate in \cite{MZ19} is direct since they were treating a Couette flow, but cannot serve our purpose well in this paper as we are treating a general shear flow near Couette flow. On the other hand, the change of coordinate in \cite{BVW18} cannot directly serve our purpose neither, as there the authors only obtain a threshold $\gamma=1/2$, and we are aimed to obtain an improved stability threshold $\gamma=1/3$.  The change of coordinate we finally adopt is the one in \eqref{changeofcoordinate}, combining the ones utilized in \cite{BGM17} and  \cite{MZ19} in spirit, as well as the one in \cite{BVW18}. After long and tedious argument, we find that this change of coordinate serves our purpose perfectly. Some estimates and the multipliers are similar to those presented in \cite{MZ19} and therefore, we will not give the details of those estimates that have already presented in their paper, but just give a sketch if necessary, to shorten the length of this paper. 

We would also like to remark that by this type of change of coordinate, we may hope to obtain a stability threshold no greater than $\gamma=3/2$ of a shear flow near the Couette flow for the 3D Navier-Stokes equation in Sobolev space, thus generalizing the result of Bedrossian, Germain and Masmoudi \cite{BGM17}. We may also improve the stability threshold to $\gamma=1/3$ of a general shear flow near the Couette flow for the 2D Boussinesq equation, whose stability was shown to be $\gamma=1/2$ for velocity field in Sobolev spaces of the Couette flow \cite{Zill20} as well as for a shear flow near the Couette flow \cite{BP21}. These are left for further work in short future.

This paper is organized as follows. In Section 2, we prove Theorem \ref{thm1.1} by assuming several estimates that will be proved in the subsequent sections. In Sections 3 to 9, we prove these estimates.

\section{Proof of Theorem \ref{thm1.1}}
\setcounter{equation}{0}

\subsection{Coordinate transform}
Set the coordinate transform
\begin{equation}\label{changeofcoordinate}
\begin{cases}
X=X(t,x,y)=x-t\vartheta(t,y)=x-t\bar U(t,y)-t\psi(t,y),\\
Y=Y(t,y)=\vartheta(t,y)=\bar U(t,y)+\psi(t,y).
\end{cases}
\end{equation}
where $\bar U(t,y)$ is the solution of the one dimensional heat equation
$$\partial_t\bar U(t,y)=\nu\partial_{yy}\bar U(t,y),\ \ \ \ \bar U(0,y)=U(y).$$
and $\psi$ in the coordinate transform satisfies the following
\begin{equation}
\begin{cases}
\displaystyle(\partial_t-\nu\partial_{yy})(t\psi)=\la v^x\ra  =\frac{1}{2\pi}\int_{\Bbb T} v^x(t,x,y)dx,  \\
\lim_{t\to0}(t\psi)=0.
\end{cases}
\end{equation}
We note that when $U(y)=y$ is the Couette flow, this change of coordinate reduces to the one employed in \cite{BGM17,MZ19} where the authors obtained stability threshold $\gamma=3/2$ for Couette flow of the 3D Navier-Stokes equation \cite{BGM17} and threshold $\gamma=1/3$ for Couette flow of the 2D Navier-Stokes equation \cite{MZ19}, and if we set $\psi\equiv0$, then this change of coordinate reduces to the one used in \cite{BVW18}, where the authors obtained stability threshold $\gamma=1/2$ for shear flow near Couette.

Set
$$F(t,X,Y)=f(t,x,y),$$
then
\begin{equation}
\begin{split}
\nabla f(t,x,y)= & \left(
      \begin{array}{c}
       \partial_xf\\
        \partial_yf \\
      \end{array}
 \right)= \left(
      \begin{array}{c}
       \partial_XF\\
        (\bar U'+\partial_y\psi)(\partial_Y-t\partial_X)F
      \end{array}
 \right)\\
 = & \left(
      \begin{array}{c}
       \partial_X^t F\\
        \partial_Y^tF
      \end{array}
 \right) =\nabla^t F(t,X,Y).
\end{split}
\end{equation}
Therefore (see \eqref{3.4})
$$\partial_y\bar U(t,y)=(\bar U'+\partial_y\psi)\partial_Y\bar U(t,Y),\ \ \ \partial_Y\bar U(t,Y)=(\bar U'+\partial_y\psi)^{-1}\partial_y\bar U(t,y).$$
Set %the linear part of $\nabla^t$,  we denote
\begin{equation}
\begin{split}
\nabla_L=  \left(
      \begin{array}{c}
       \partial_X\\
        \partial_Y-t\partial_X
      \end{array}
 \right) =  \left(
      \begin{array}{c}
       \partial_X\\
        \partial_Y^L
      \end{array}
 \right),
\end{split}
\end{equation}
\begin{equation}
\begin{split}
\Delta_L= \nabla_L\cdot\nabla_L=  \partial_X^2+(\partial_Y^L)^2,
\end{split}
\end{equation}
$$\widetilde {\Delta_t}=\partial_{XX}+(\bar U'+\partial_y\psi)^2\partial_{YY}^L,$$
and
\begin{equation}
\begin{split}
G=(\bar U'+\partial_y\psi)^2 -(\bar U')^2.
\end{split}
\end{equation}
The Laplacian operator $\Delta$ transforms into
\begin{equation}
\begin{split}
\Delta f= \Delta_t F=\left( (\partial_X)^2+(\partial_Y^t)^2\right)F =\Delta_LF+G\partial_{YY}^LF+\Delta_tC\partial_Y^LF+b\partial_Y^LF,
\end{split}
\end{equation}
where $\Delta_t C=\partial_{yy}\psi$ and $b=\bar U''$. It is also obvious that for a function $F(t,Y)$ independent of $X$, one has $$\widetilde {\Delta_t}F=(\bar U'+\partial_y\psi)^2\partial_{YY}F.$$

Set $\Omega(t,X,Y)=\omega(t,x,y)$ and $C(t,Y)=\psi(t,y)$, then 
\begin{equation}\label{e2.18}
\begin{cases}
\Omega_t+u(t,X,Y)\cdot\nabla_{X,Y}\Omega= \bar U''\partial_X\phi +\nu\widetilde {\Delta_t}\Omega,\\
\Delta_t\phi=\Omega,\\
u(t,X,Y)=\left(
      \begin{array}{c}
         0\\
         g\\
      \end{array}
 \right) +(\bar U'+\partial_y\psi) \nabla^{\bot}_{X,Y}\phi_{\ne},
\end{cases}
\end{equation}
where
$$\bar U'(t,Y)=\partial_y\bar U(t,y), \ \ \ \bar U''(t,Y)=\partial_{yy}\bar U(t,y),$$
and
\begin{equation}
\begin{split}
g=\frac1t(V^X_0(t,Y)-C(t,Y)).
\end{split}
\end{equation}
As a function of $(t,Y)$, $g$ evolves according to the following
\begin{equation}\label{gg}
\begin{split}
g_t+\frac{2g}t+g\partial_Yg=-\frac{(\bar U'+\partial_y\psi)}t\langle\nabla^{\bot}_{X,Y}\phi_{\ne}\cdot\nabla_{X,Y}V^X\rangle+  \nu\widetilde {\Delta_t}g.
\end{split}
\end{equation}
Set $h=\partial_y\psi(t,y)$, then $h$, as a function of $(t,Y)$, evolves according to
\begin{equation}\label{hh}
\begin{split}
h_t+g\partial_Yh=\bar h+ \nu\widetilde {\Delta_t}h,
\end{split}
\end{equation}
where $\bar h=-\frac{f_0+h}{t}$ satisfies
\begin{equation}\label{barh}
\begin{split}
\bar h_t+\frac2t\bar h+g\partial_Y\bar h=\frac{(\bar U'+\partial_y\psi)}t\langle\nabla^{\bot}_{X,Y}\phi_{\ne}\cdot\nabla_{X,Y}\Omega\rangle+  \nu\widetilde {\Delta_t}\bar h.
\end{split}
\end{equation}
For later purpose, we set $V^X(t,X,Y)=v^x(t,x,y)$, and by the Biot-Savart law, the $X$-component of velocity can be expressed as   
\begin{equation}\label{u1}
\begin{split}
V^X=-(\bar U'+\partial_y\psi)\partial_Y^L\phi= -(\bar U'+\partial_y\psi)(\partial_Y-t\partial_X)\phi.
\end{split}
\end{equation}
%{\color{red}\begin{equation}\label{u1}
%\begin{split}
%U^1=-(\bar U'+\partial_y\psi)\partial_Y^L\phi= -(\bar U'+\partial_y\psi)(\partial_Y-t\partial_X)\phi.    \ \ \ \ \  (delete!)
%\end{split}
%\end{equation}
%}

%\subsection{Notations}
%We summarize the notations used in this paper. 
%\begin{equation}
%\begin{split}
%\Omega(t,X,Y)=& \omega(t,x,y)\\
%C(t,Y)=& \psi(t,y)\leftrightarrow v(t,y)-y,\ \  (use:\ \vartheta(t,y)-y?)\\
%h(t,Y)=& \partial_y\psi(t,y),\ \ \ \ \ h'(t,Y)=\\
%\bar U'(t,Y)= & \partial_y\bar U(t,y)\\
%\bar U''(t,Y)= & \partial_{yy}\bar U(t,y)\\
%\phi(t,X(t,x,y),Y(t,y))=  & \varphi(t,x,y)\ \ use \ \Delta_t\varphi(t,x,y)=\Omega?\\
%\tilde v(t,X(t,x,y),Y(t,y))=  & v^{x}(t,x,y), \ \ i.e., U^1\ in \ \eqref{u1}?\\ 
%should\ not\ be\ confused\  & with\ the \ following\ coordinate\ change\ v(t,Y)
%\end{split}
%\end{equation}
%with correspondence to \cite{MZ19} (use $\vartheta(t,y)$ to replace $v(t,y)$ in the following?)
%\begin{equation}
%\begin{split}
%v'(t,Y) \leftrightarrow & \bar U'+\partial_y\psi,\\
%v''(t,Y) \leftrightarrow & \bar U''+ \partial_{yy}\psi =\bar U''+ \partial_yY\partial_{Y}h =\bar U''+v'(t,Y)\partial_{Y}h 
%\end{split}
%\end{equation}
%
%
%Maybe we can use $\psi\to\Psi,\varPsi,\phi\to\Phi,\varPhi$ 
%
%We usually omit the subscript 2 of the $L^2$-norm $\|\cdot\|_2$ by simply writting $\|\cdot\|$.

\subsection{Multipliers and notations}

The Fourier variable of a function $f(X,Y)$ is taken to be $(k,\eta)\in \Bbb Z\times \Bbb R$, i.e.,
$$\mathcal F(f)(k,\eta)= \widehat f(k,\eta)= \widehat{f}_k(\eta) =\frac{1}{2\pi}\int_{\Bbb T\times \bf R}e^{-i(kX+\eta Y)}f(X,Y)dXdY,$$
with Fourier inverse formula
$$f(X,Y)=\frac{1}{2\pi}\sum_{k\in\Bbb Z}\int_{\bf R}\widehat{f}_k(\eta)e^{i(kX+\eta Y)}d\eta.$$
In this way, we define $|k,\eta|=|k|+|\eta|$ and $\la \xi\ra=(1+|\xi|^2)^{1/2}$ for a scalar or a vector $\xi\in\Bbb R^n$. The $x$-average of a function $\la f\ra=f_0=\frac{1}{2\pi}\int f(x,y)dx$ and $f_{\neq}:=f-f_0$. The standard $H^{\sigma}$ norm is given by $\|f\|_{H^{\sigma}}=\|\la \xi\ra^{\sigma}\widehat f\|_{L^2}$ and the space-time Sobolev space $L^p_TH^{\sigma}$ is defined in the standard way. 

Set $\Bbb D=\{1/2,1,2,4,8,\cdots,2^j,\cdots\}$. Set $\sigma\geq40$ and define the time dependent norm as
$$\|A^{\sigma}(t,\nabla)\Omega\|^2_2=\sum_{k}\int_{\eta} |A_k^{\sigma}(t,\eta)\hat\Omega_k(t,\eta)|^2d\eta,\ \ \ A_k^{\sigma}(t,\eta)=\frac{\la k,\eta\ra^{\sigma}}{w_k(t,\eta)},$$
where the definition of $w_k(t,\eta)$ is lengthy and is outlined in Appendix. By definition of $w_k(t,\eta)$, we have $A^{\sigma}_k(t,\eta)\approx \la k,\eta\ra^{\sigma}$. Introduce another time dependent norm for $8\leq s\leq \sigma-10$,
$$\|A_E^{s}(t,\partial_k,\partial_Y)\Omega\|^2_2=\sum_{k\neq0}\int_{\eta} |A_E^{s}(t,k,\eta)\hat\Omega_k(t,\eta)|^2d\eta,$$
which quantifies the enhanced dissipation effect with $A_E^{s}(t,k,\eta)=\la k,\eta\ra^s D(t,\eta)$ and $D(t,\eta)=\frac{1}{3}\nu|\eta|^3+ \frac{1}{24}\nu(t^3-8|\eta|^3)_+.$

%We also define other important multipliers. All the multipliers come from the paper \cite{MZ19}. By definition, we have Lemma 6.1 in \cite{MZ19}. We will need the following lemma concerning the estimate of the multiplier $w(t,\eta)$. From Lemma 6.1 in \cite{MZ19}, we know that when $|\xi-\eta|\leq |\eta|/10$, then it holds
%\begin{equation}\label{lem6.1}
%|w(t,\eta)-w(t,\xi)|\lesssim
%\frac{|\xi-\eta|}{\la \eta\ra}\times \begin{cases}
%\nu^{-1/3},\ \ \ \ \ t\lesssim \nu^{-1/3}\\
%\nu^{\beta/3}t^{1-\beta},\ \ \ \ \ \gtrsim \nu^{-1/3}.
%\end{cases}
%\end{equation}

\subsection{Main energy estimate}
For the systems \eqref{e2.18}, \eqref{gg} and \eqref{barh}, we define the Sobolev energy
$$\mathcal E^{\sigma}(t)=\frac12\|A^{\sigma}(t)\Omega(t)\|^2_2+\mathcal E_v(t),$$
with
$$\mathcal E_v(t)=\|g(t)\|_{H^{\sigma}}^2+ \nu^{1/3}\|h(t)\|_{H^{\sigma}}^2+ \nu^{1/3}\|\bar h(t)\|_{H^{\sigma}}^2+ \|h(t)\|_{H^{\sigma-1}}^2+ \|\bar h(t)\|_{H^{\sigma-1}}^2.$$
We have the following Lemma, due to the well-posedness theory for the 2D Navier-Stokes  equation in Sobolev spaces, and we can ignore the time interval $[0,1]$ by further restricting the size of the initial data. 
\begin{lemma}
For $\epsilon>0,\nu>0$ and $\sigma\geq40$, there exists $\epsilon'>0$ such that if $\|\Omega_{in}\|_{H^{\sigma}}\leq \epsilon'\nu^{1/3}$, then $$\sup_{t\in[0,1]}\mathcal E^{\sigma}(t)\leq (\epsilon\nu^{1/3})^2.$$
\end{lemma}

In the sequel, we set the following system of bootstrap hypotheses for $t\geq1$.

\begin{itemize}
  \item {Higher regularity for the main system}
\begin{equation}\label{e2.22}
\begin{split}
\|A^{\sigma}\Omega(t)\|^2_2+\nu\int_1^t \|\sqrt{-\Delta_L}A^{\sigma}\Omega(\tau)\|^2_2d\tau +\int_1^tCK_w(\tau)d\tau \leq (8\epsilon\nu^{1/3})^2,
\end{split}
\end{equation}
where
$$CK_w(t):=\sum_{k}\int \frac{\partial_tw_k(t,\eta)}{w_k(t,\eta)} |A_k^{\sigma}(t,\eta)\hat\Omega_k(t,\eta)|^2d\eta.$$
\end{itemize}

\begin{itemize}
  \item  {Higher regularity for the coordinate system}
\begin{equation}\label{BS2.2}
\begin{split}
&\la t\ra\|g\|_{H^{\sigma}}+\int_1^t\|g(\tau)\|_{H^{\sigma}} d\tau \leq 8\epsilon\nu^{1/3},\\
& t^3\|A^{\sigma}\bar h(t)\|^2_2+\int_1^t \tau^3\|\sqrt{\frac{\partial_tw}{w}}A^{\sigma}\bar h(\tau)\|^2_2d\tau +\int_1^tCK_w(\tau)d\tau\\
&\ \ \ \ \ \ \ \ \ \ \ \ \ \ \ \  +\frac14\int_1^t\tau^2\|A^{\sigma}\bar h(\tau)\|^2_2d\tau +\frac\nu4\int_1^t\tau^3\|\partial_YA^{\sigma}\bar h(\tau)\|^2_2d\tau\leq 8\epsilon(\epsilon\nu^{1/6})^2,\\
& \|h(t)\|^2_{H^{\sigma}} +\nu\int_1^t\|\partial_Yh(\tau)\|^2_{H^{\sigma}}d\tau\leq 8(10\epsilon\nu^{1/6})^2.
\end{split}
\end{equation}
\end{itemize}

\begin{itemize}
  \item  {Lower regularity: enhanced dissipation}
\begin{equation}
\begin{split}
\|A_E^{s}\Omega(t)\|^2_2+\frac25\nu\int_1^t \|\sqrt{-\Delta_L}A_E^{s}\Omega(\tau)\|^2_2d\tau \leq (8\epsilon\nu^{1/3})^2.
\end{split}
\end{equation}
\end{itemize}

\begin{itemize}
  \item   {Lower regularity: decay of zero mode}
\begin{equation}\label{e2.25}
\begin{split}
&\la t\ra^4\|g\|^2_{H^{\sigma-6}} +\nu\int_1^t\tau^4\|\partial_Yg(\tau)\|^2_{H^{\sigma-6}} d\tau \leq (8\epsilon\nu^{1/3})^2,\\
& \la t\ra^4\|\bar h(t)\|^2_{H^{\sigma-6}} +\nu\int_1^t\tau^4\|\partial_Y\bar h(\tau)\|^2_{H^{\sigma-6}} d\tau \leq (8\epsilon\nu^{1/3})^2, \\
& \|\Omega_0\|^2_{H^{s}} +\frac{t\nu}{2}\|\partial_Y\Omega_0\|^2_{H^{s}}  +\nu\int_1^t\left(\|\partial_Y\Omega_0(\tau)\|^2_{H^{s}} +\frac{\tau\nu}{2}\|\partial_Y\Omega_0(\tau)\|^2_{H^{s}}\right) d\tau \leq (8\epsilon\nu^{1/3})^2.
\end{split}
\end{equation}
\end{itemize}

\begin{itemize}
  \item   {Assistant estimate}
\begin{equation}
\begin{split}
& \la t\ra\|\bar h(t)\|_{H^{\sigma-1}} +\int_1^t\|\bar h(\tau)\|_{H^{\sigma-1}} d\tau \leq 8\epsilon\nu^{1/3}, \\
& \|h(t)\|^2_{H^{\sigma-1}} +\nu\int_1^t\|\partial_Yh(\tau)\|^2_{H^{\sigma-1}} d\tau \leq 8(10\epsilon\nu^{1/3})^2.
\end{split}
\end{equation}
\end{itemize}

The following proposition follows from the hypotheses, elliptic estimates and the properties of the multipliers $A^s_E$ and $A^{\sigma}$.
\begin{prop}
Under the bootstrap hypotheses, the following inequalities hold
\begin{equation}
\begin{split}
&   \|\Omega\|_{H^{\sigma}}  +\nu^{1/2} \|\sqrt{-\Delta_L}\Omega\|_{L^2_TH^{\sigma}}  +\|\sqrt{\frac{\partial_tw}{w}}\Omega\|_{L^2_TH^{\sigma}}  \lesssim \epsilon\nu^{1/3}, \\
&   \|\Omega_{\neq}\|_{H^{s}}  +\nu^{1/2} \|\sqrt{-\Delta_L}\Omega_{\neq}\|_{L^2_TH^{s}}  \lesssim \frac{\epsilon\nu^{1/3}}{\la \nu t^3\ra},\ \ \\
&   \|\phi_{\neq}\|_{H^{\sigma-4}}\lesssim     \frac{\epsilon\nu^{1/3}}{\la t^2\ra}, \ \ \ \ \  \|V^X_{\neq}\|_{H^{\sigma-3}}\lesssim \frac{\epsilon\nu^{1/3}}{\la t\ra}.
\end{split}
\end{equation}
\end{prop}
This can be proved as in \cite{MZ19}, by noting that $A_k^{\sigma}(t,\eta)\approx \la k,\eta\ra^{\sigma}$, $D(t,\eta)\geq \nu t^3$, Lemma \ref{lem3.1} and Lemma \ref{lem4.2}.

Using this result and the note on \cite[P.4]{MZ19}, we can prove Theorem \ref{thm1.1}. To employ the continuation method, we will show that under these bootstrap hypotheses, these bounds can be improved by a even smaller bounds. We will prove the following Proposition.
\begin{prop}
For $\sigma\geq40$, $\nu>0$ and $8\leq s\leq \sigma-10$, there exists $0<\epsilon_0,\nu_0<1$, such that for all $0<\nu<\nu_0$ and $0<\epsilon<\epsilon_0$, such that the bootstrap hypotheses implies themselves with $8$ replaced by $6$.
\end{prop}
\begin{proof}
In the rest of this section, we give an outline of the proof of the higher regularity for the main system. From the evolution of $\mathcal E_{H,\Omega}=\frac12\|A^{\sigma}\Omega(t)\|^2$, we have
\begin{equation}\label{eq2.1}
\begin{split}
\frac12\frac{d}{dt}\int_{\Bbb T\times\bf R}   &   |A^{\sigma}\Omega(t)|^2dXdY     +\underbrace{\sum_{k}\int\frac{\partial_tw_k(t,\eta)} {w_k(t,\eta)} |A^{\sigma}_k(t,\eta)\hat\Omega_k(t,\eta)|^2d\eta}_{CK_{w}}     =  \nu\int A^{\sigma}\Omega A^{\sigma}(\tilde\Delta_t\Omega)dXdY\\
&  -\int A^{\sigma}\Omega A^{\sigma}(u\cdot\nabla_{X,Y}\Omega)dXdY    +\int A^{\sigma}\Omega A^{\sigma}(\bar U''\partial_X\phi)dXdY= DE(t) -CV(t) +S(t),
\end{split}
\end{equation}
%{\centerline{\color{red}[Here, $u$ should be replaced with $v$?  \eqref{e2.18}, and coordinate $v$ be replaced with $\vartheta$.]}}
where the first term on the r.h.s. is usually referred to as the \emph{Cauchy-Kovalevshaya} term. {We aim to show that after integration in time $[1,t]$, it holds
%\begin{equation}
%\begin{split}
%\|A^{\sigma}\Omega(t)\|_2^2- \|A^{\sigma}\Omega(1)\|_2^2 +\int_1^t CK_{w}(t) d\tau\lesssim main+ C\epsilon\sup_{\tau\in[1,t]} \|A^{\sigma}\Omega(\tau)\|_2^2 +\cdots
%\end{split}
%\end{equation}
\begin{equation*}
\begin{split}
\|A^{\sigma}\Omega(t)\|^2_2 & +2\int_1^tCK_{w}(\tau)d\tau +\frac74\int_1^t\nu\|\sqrt{-\Delta_L}A^{\sigma} \Omega\|^2_2d\tau \\
\lesssim & \|A^{\sigma}\Omega(1)\|^2_2  + C\epsilon \sup_{\tau\in [1,t]}\|A^{\sigma}\Omega(\tau)\|^2_2 +C\epsilon\int_1^tCK_{w}(\tau)d\tau   +C\epsilon^{3}\nu^{2/3}.
\end{split}
\end{equation*}
}

First, for the last \emph{Dissipation Error} term $DE(t)$  in \eqref{eq2.1}, we have
\begin{equation}\label{e2.45}
\begin{split}
\nu\int A^{\sigma}\Omega & A^{\sigma}(\tilde\Delta_t\Omega)dXdY \\
= & \nu\int A^{\sigma}\Omega A^{\sigma}(\Delta_L\Omega)dXdY -\nu\int A^{\sigma}\Omega A^{\sigma}[(1-(\bar U'+\partial_y\psi)^2)\partial_{YY}^L\Omega]dXdY\\
= & -\nu\|\sqrt{-\Delta_L}A^{\sigma}\Omega\|_2^2 -\nu\int A^{\sigma}\Omega_0 A^{\sigma}[(1-(\bar U'+\partial_y\psi)^2)\partial_{YY}\Omega_0]dXdY\\
& -\nu\int A^{\sigma}\Omega_{\neq} A^{\sigma}[(1-(\bar U'+\partial_y\psi)^2)(\partial_{Y}-t\partial_X)^2\Omega_{\neq}]dXdY\\
= & -\nu\|\sqrt{-\Delta_L}A^{\sigma}\Omega\|^2_2 +E^0 +E^{\neq}.
\end{split}
\end{equation}
We can show that
%\begin{equation} 
%\begin{split}
%\left|  \nu\int A^{\sigma}\Omega  A^{\sigma}(\tilde\Delta_t\Omega)dXdY  \right| \lesssim  ...
%\end{split}
%\end{equation}
\begin{equation}\label{prop2.5}
\begin{split}
\nu\int_1^t\left| \int A^{\sigma}\Omega A^{\sigma}(\tilde\Delta_t\Omega)dXdY\right| d\tau \leq -\frac78\int_1^t\nu\|\sqrt{-\Delta_L}A^{\sigma} \Omega\|^2_2d\tau +C\max\{\nu^{2/3},\epsilon\nu{1/2},\epsilon\}\cdot \epsilon^2\nu^{2/3}.
\end{split}
\end{equation}

For the \emph{convective} term $CV(t)$ in \eqref{eq2.1}, we have
\begin{equation}\label{e2.44}
\begin{split}
CV(t)=  -\frac12\int \nabla\cdot u|A^{\sigma}\Omega|^2dXdY   +\underbrace{ \int A^{\sigma}\Omega \Big(A^{\sigma}(u\cdot\nabla_{X,Y}\Omega -u\cdot\nabla_{X,Y}A^{\sigma}\Omega)\Big)dXdY}_{CM(t)},
\end{split}
\end{equation}
where $\nabla_{X,Y}\cdot u=\partial_Yg +\partial_X\phi_{\neq}\partial_Y(\bar U'+h)$ by definition \eqref{e2.18}$_3$.  For the first part, we have
\begin{equation*}
\begin{split}
\left|  \int \nabla\cdot u|A^{\sigma}\Omega|^2dXdY\right|   \lesssim & \|\nabla u\|_{L^{\infty}}\|A^{\sigma}\Omega\|^2_{2}\\
\lesssim & \left(\|g\|_{H^2}+ (1+\|h\|_{H^2}+\|\bar U'\|_{H^2})\|\phi_{\neq}\|_{H^3}\right) \|A^{\sigma}\Omega\|^2_{2}\\
\lesssim & \frac{\epsilon\nu^{1/3}}{\la t\ra^2}\|A^{\sigma}\Omega\|^2_{2},
\end{split}
\end{equation*}
where we have used the bootstrap hypotheses \eqref{e2.25} and decay estimate $\|\phi_{\neq}\|_{H^{\sigma-4}}\lesssim {\epsilon\nu^{1/3}}{\la t\ra^{-2}}$ in Lemma \ref{lem3.1} below.  For the commutator term in \eqref{e2.44}, we use a paraproduct decomposition
\begin{equation}\label{e2.44}
\begin{split}
CM(t)=  \int A^{\sigma}\Omega \left[A^{\sigma}(u\cdot\nabla_{X,Y}\Omega -u\cdot\nabla_{X,Y}A^{\sigma}\Omega)\right] dXdY=\sum_{N\geq8}T_N+\sum_{N\geq8}R_N+\mathcal R,
\end{split}
\end{equation}
where
\begin{equation}\label{e2.44-1}
\begin{split}
T_N= &  \int A^{\sigma}\Omega \left[A^{\sigma}(u_{<N/8}\cdot\nabla_{X,Y}\Omega_N -u_{<N/8}\cdot\nabla_{X,Y}A^{\sigma}\Omega_N)\right] dXdY,      \\
R_N= &  \int A^{\sigma}\Omega \left[A^{\sigma}(u_N\cdot\nabla_{X,Y}\Omega_{<N/8}) -u_N\cdot\nabla_{X,Y}A^{\sigma}\Omega_{<N/8}\right] dXdY,  \\
\mathcal R= & \sum_{N\in\Bbb D;N/8\leq N'\leq 8N}\int A^{\sigma}\Omega \left[A^{\sigma}(u_N\cdot\nabla_{X,Y}\Omega_{N'} -u_N\cdot\nabla_{X,Y}A^{\sigma}\Omega_{N'})\right] dXdY.
\end{split}
\end{equation}
Here, $\displaystyle N\in \Bbb D=\{\frac12,1,2,4,\cdots,2^j,\cdots\}$.  For the commutator term  $CM(t)$, we can show that
\begin{equation}\label{prop2.678}
\begin{split}
\int_1^tCM(\tau)d\tau \lesssim \epsilon \sup_{\tau\in [1,t]}\|A^{\sigma}\Omega(\tau)\|^2_2 +\epsilon\int_1^tCK_{w}(\tau)d\tau +\epsilon^5\nu^{\frac23}.
\end{split}
\end{equation}

For the \emph{source} term $S(t)$ in \eqref{eq2.1}, we use again the Bony decomposition
\begin{equation}
\begin{split}
\bar U''\partial_X\phi=T_{\bar U''}(\partial_X\phi) +T_{\partial_X\phi}{(\bar U'')}+\mathcal R(\bar U'',\partial_X\phi)
\end{split}
\end{equation}
to decompose
\begin{equation}
\begin{split}
S(t):=\int\int & A^{\sigma}\Omega A^{\sigma}(\bar U''\partial_X\phi)dXdY =S^{LH}+S^{HL}+S^{R},
\end{split}
\end{equation}
where   %{\color{red}(WHY do we do this decomposition??)}
\begin{equation}
\begin{split}
S^{LH}= & 2\pi\sum_{N\geq8}\int  A^{\sigma}\Omega A^{\sigma}({\bar U''}_{<N/8}\partial_X\Delta_t^{-1}\Omega_{N})dXdY,\\
S^{HL}= & 2\pi\sum_{N\geq8}\int  A^{\sigma}\Omega A^{\sigma}({\bar U''}_N\partial_X\Delta_t^{-1}\Omega_{<N/8})dXdY,\\
S^R= & 2\pi\sum_{N\in\Bbb D}\sum_{\frac{N}{8}\leq N'\leq 8N}\int  A^{\sigma}\Omega A^{\sigma}({\bar U''}_N\partial_X\Delta_t^{-1}\Omega_{N'})dXdY.
\end{split}
\end{equation}
%We can show that 
%\begin{equation}
%\begin{split}
%\int_1^tS(\tau)d\tau \lesssim & \int_1^t \|\bar U''\|_{H^{\sigma}}\|A^{\sigma}\Omega_{\neq}\|^2_{L^2} + \sup_{\tau\in[1,t]}\|\bar U''\|_{H^{\sigma}}\int_1^t{\epsilon^2\nu^{2/3}}\|\partial_Yh\|^2_{H^{\sigma}}d\tau\\
%\lesssim  & C\epsilon \|A^{\sigma}\Omega_{\neq}\|^2_{L^2}  +C(\epsilon\nu^{1/3})^2,
%\end{split}
%\end{equation}
It can be shown that for any $t\geq1$, it holds
$$\int_1^tS(\tau)d\tau \leq C\epsilon\sup_{\tau\in[1,t]} \|A^{\sigma}\Omega(\tau,\cdot)\|^2_{2}+\epsilon^4\nu^{7/12}(8\epsilon\nu^{1/3})^2.$$
under the bootstrap assumption \eqref{BS2.2} and the smallness assumption of the background Couette flow that $ \|U(y)-y\|_{H^{\sigma}}\lesssim \nu^{1/3}$.

%Hence for $\delta$ sufficiently small, this term is absorbed by the $CK_{w}$ term of \eqref{eq2.1}. 

Putting these estimates together, we finally arrive at
\begin{equation}%\label{e25}
\begin{split}
\|A^{\sigma}\Omega(t)\|^2_2 & +2\int_1^tCK_{w}(\tau)d\tau +\frac74\int_1^t\nu\|\sqrt{-\Delta_L}A^{\sigma} \Omega\|^2_2d\tau \\
\lesssim & \|A^{\sigma}\Omega(1)\|^2_2  + C\epsilon \sup_{\tau\in [1,t]}\|A^{\sigma}\Omega(\tau)\|^2_2 +C\epsilon\int_1^tCK_{w}(\tau)d\tau   +C\epsilon^{3}\nu^{2/3}.
\end{split}
\end{equation}
By taking $\epsilon$ sufficiently small, we complete the proof.

Other bootstrap hypotheses are closed in Sections \ref{Coordinate} and \ref{DecayVorticity}.
\end{proof}

\section{Elliptic estimate}
\setcounter{equation}{0}
This section provides some analyses of $\Delta_t$. Recall that
$$\bar U(t,y)=e^{\nu t\partial_{yy}}U(y)=y+e^{\nu t\partial_{yy}}(U(y)-y)=y+\widetilde U(t,y).$$
Below, we require that the shear flow $U(y)$ is close to the Couette flow  in the sense $\|U(y)-y\|_{H^{\sigma}}\leq \epsilon\nu^{1/3}$.
\begin{lemma}\label{lem3.1}
Under the bootstrap hypotheses, for $\nu$ sufficiently small and $s'\in [0,2]$, it holds that for $2\leq \gamma\leq \sigma-1$
$$\|\phi_{\ne}\|_{H^{\gamma-s'}}\lesssim \frac{1}{\langle t\rangle^{s'}}\|\langle\partial_X\rangle^{-s'}\Omega_{\ne}\|_{H^{\gamma}},$$
and for $\gamma\leq \sigma-1$
$$\|\Delta_L\Delta_t^{-1}\Omega_{\ne}\|_{H^{\gamma}}=\|\Delta_L\phi_{\ne}\|_{H^{\gamma}}\lesssim \|\Omega_{\ne}\|_{H^{\gamma}}.$$
\end{lemma}
\begin{proof}
The first one is obvious by noting that $\la t\ra^{2s'}\lesssim \la \eta/k\ra^{2s'}(k^2+|\eta-kt|^2)^2$. For the second one, note that we can write $\Delta_t$ as a perturbation of $\Delta_L$ via
$$\Delta_L\phi_{\neq}=\Delta_t\phi_{\neq} +(1-(\bar U'+\partial_y\psi)^2)\partial_{YY}^L\phi_{\neq} -(\bar U''+\partial_{yy}\psi)\partial_Y^L\phi_{\neq}.$$
Under the bootstrap assumption \eqref{BS2.2} and the smallness assumption of $\|U-y\|_{H^{\sigma}}\lesssim \epsilon\nu^{1/3}$, we have
\begin{equation}
\begin{split}
\|1-(\bar U'+\partial_y\psi)^2\|_{H^{\gamma}}\leq \|1+(\bar U'+\partial_y\psi)\|_{H^{\gamma}}\|(\bar U'-1)+\partial_y\psi\|_{H^{\gamma}} \lesssim \epsilon\nu^{1/6}.
\end{split}
\end{equation}
The other steps follow from the proof of Lemma 4.1 in \cite{MZ19}, since one has 
\begin{equation}
\begin{split}
\|\Delta_L\phi_{\neq}\|_{H^{\gamma}} \leq \|\Omega_{\neq}\|_{H^{\gamma}} + C\epsilon\nu^{1/6}\|\Delta_L\phi_{\neq}\|_{H^{\gamma}}.
\end{split}
\end{equation}
The result then follows by setting $\epsilon$ sufficiently small and letting $s=\sigma-2-s'$.
\end{proof}

\begin{lemma}\label{lem4.2}
Under the bootstrap hypotheses, it holds that
$$\|\nabla_L\phi_{\ne}\|_{H^{\sigma-2}}+ \|V^X_{\neq}\|_{H^{\sigma-2}} \lesssim \frac{1}{\la t\ra}\|\Omega_{\neq}\|_{H^{\sigma-1}},$$
and for $\gamma\leq \sigma-1$
$$\|\nabla_LV^X_{\neq}\|_{H^{\gamma}}\lesssim \|\Omega_{\neq}\|_{H^{\gamma}},$$
where $V^X$ is defined in \eqref{u1} via the Biot-Savart law.
\end{lemma}
\begin{proof}
By definition, we have
$$V^X_{\neq} =-(\bar U'+\psi_y)(\partial_Y-t\partial_X)\phi_{\neq},$$
and hence
\begin{equation}
\begin{split}
\|V^X_{\neq}\|_{H^{\sigma-2}} \lesssim & (1+\|e^{\nu t\partial_{yy}}(U(y)-y)\|_{H^{\sigma-2}}+\|\psi_y\|_{H^{\sigma-2}})\|(\partial_Y-t\partial_X)\phi_{\neq}\|_{H^{\sigma-2}}\\
\lesssim & \frac{(1+\|e^{\nu t\partial_{yy}}(U(y)-y)\|_{H^{\sigma-2}}+\|\psi_y\|_{H^{\sigma-2}})}{\langle t\rangle} \|\Delta_L\phi_{\neq}\|_{H^{\sigma-1}}\\
\lesssim & \frac{1}{\langle t\rangle}\|\Omega_{\neq}\|_{H^{\sigma-1}},
\end{split}
\end{equation}
where we have set $\widetilde U(t,y)=e^{\nu t\partial_{yy}}(U(y)-y)$ and $U(y)-y$ is suitably small.
Similarly, we have
$$\partial_XV^X_{\neq}=-(\bar U'+\psi_y)(\partial_Y-t\partial_X)\partial_X\phi_{\neq},$$
and
$$(\partial_Y-t\partial_X)V^X_{\neq} =-(\bar U'+\psi_y)(\partial_Y-t\partial_X)^2\phi_{\neq} -\partial_Y(\bar U'+\psi_y)(\partial_Y-t\partial_X)\phi_{\neq}.$$
Note that by change of variable, we have
\begin{equation}\label{3.4}
\begin{split}
\partial_Y=t\partial_x+(\bar U'+\psi_y)^{-1}\partial_y,\ \ \ \partial_Y\bar U'=(\bar U'+\psi_y)^{-1}\bar U'' =(1+\widetilde U'+\psi_y)^{-1}\bar U''.
\end{split}
\end{equation}
Therefore, by bootstrap hypotheses, we have
\begin{equation}
\begin{split}
\|\nabla_LV^X_{\neq}\|_{H^{\gamma}}\lesssim (1+\|\bar U'\|_{H^{\sigma-1}}+\|h\|_{H^{\sigma-1}})\|\Delta_L\phi_{\neq}\|_{H^{\gamma}},
\end{split}
\end{equation}
where we have used the $\|(1+\varepsilon)^{-1}\|_{*}\lesssim 1+ \|\varepsilon\|_*$ type estimate.
\end{proof}

Similar to that in \cite{MZ19}, we have the following Lemma.
\begin{lemma}\label{lem3.3}
Under the bootstrap hypotheses, it holds that
$$\|\nabla_LV^X_{\neq}\|_{H^{\sigma}}\lesssim \|\Delta_L\Delta_t^{-1}\Omega_{\neq}\|_{H^{\sigma}} \lesssim \|\Omega_{\neq}\|_{H^{\sigma}} +\frac{\epsilon\nu^{1/3}}{\la t\ra\la \nu t^3\ra}\|\partial_Yh\|_{H^{\sigma}},$$
and
$$\|\sqrt{\frac{\partial_tw}{w}}\chi_R\Delta_L\Delta_t^{-1}\Omega_{\neq}\|_{H^{\sigma}} \lesssim \|\sqrt{\frac{\partial_tw}{w}}\Omega_{\neq}\|_{H^{\sigma}} +\frac{\epsilon^2\nu^{1/2}}{\la \nu t^3\ra}.$$
\end{lemma}
This Lemma can be proved by noting that $\|1-(\vartheta'(t,y))^2\|_{H^3}\leq \epsilon\nu^{1/3}$ and $\|1-(\vartheta'(t,y))^2\|_{H^{\sigma}}\leq \epsilon\nu^{1/6}$ by bootstrap hypotheses.  As a corollary, we have from the bootstrap hypotheses that $u$ in \eqref{e2.18} can be bounded by
$$\|u\|_{H^s} \lesssim \|g\|_{H^s}  + (\|\bar U'\|_{H^s} +\|h\|_{H^s}) \|\phi_{\neq}\|_{H^{s+1}}   \lesssim \frac{\epsilon\nu^{1/3}}{\la t\ra^2}.$$

\begin{lemma}[\cite{MZ19}]\label{lem4.4}
Under the bootstrap hypotheses, it holds  that for $\epsilon$ sufficiently small and for $s\leq \sigma-7$ that
$$\|A^s_E(\phi_{\neq})\|_{2}\lesssim \frac{1}{\la t\ra^2}(\|A^s_E\Omega\|_{2}+ \|\Omega\|_{H^{\sigma}}).$$
\end{lemma}
The proof of the above two Lemmas can be found in \cite{MZ19}.
\begin{lemma}\label{lem4.5}
Under the bootstrap hypotheses, it holds  that for $\epsilon$ sufficiently small and for $s\leq \sigma-7$ that
$$\|A^s_E\phi_{\neq}\|_{2}\lesssim \nu \|\phi_{\neq}\|_{s+3}+ \|\nu t^3\phi_{\neq}\|_{H^{s}}).$$
\end{lemma}
\begin{proof}
By definition, it is obvious that 
$$\nu\max\{t^3,|\eta|^3\}\lesssim D(t,\eta)\lesssim \nu\max\{t^3,|\eta|^3\}.$$
Therefore
\begin{equation}
\begin{split}
\|A^s_E\phi_{\neq}\|\lesssim & \nu\|\la k,\eta\ra^{s}|\eta|^3\phi_{\neq}\| +\nu t^3\|\phi_{\neq}\|_{H^s}\\
\lesssim & \nu\||\phi_{\neq}\|_{H^{s+3}} +\nu t^3\|\phi_{\neq}\|_{H^s}.
\end{split}
\end{equation}
The result follows.
\end{proof}

\section{Estimate of the dissipation error term}
\setcounter{equation}{0}

In this section, we estimate the \emph{Dissipation Error} term in \eqref{eq2.1}.  Set $\vartheta(t,y)=\bar U(t,y)+\psi(t,y)$ and  recall
\begin{equation}\label{DE}
\begin{split}
DE(t)=\nu\int A^{\sigma}\Omega A^{\sigma}(\tilde\Delta_t\Omega)dXd  =   -\nu\|\sqrt{-\Delta_L}A^{\sigma}\Omega\|^2_2 +E^0 +E^{\neq}.
\end{split}
\end{equation}
In what follows, we separate the estimate into the zero and non-zero mode cases.

\underline{Estimate of the Zero mode.} By splitting in the frequency space and the fact that $A_0^{\sigma}(\eta)\approx \la \eta\ra^{\sigma}\approx 1+|\eta|^{\sigma}$ and $|\eta|\leq |\xi|+|\eta-\xi|\lesssim \max\{|\xi|,|\eta-\xi|\}$, we have
\begin{equation}\label{e5.1}
\begin{split}
|E_0|= & \left|\nu\int A^{\sigma}\Omega_0 A^{\sigma}[(1-(\bar U'+\partial_y\psi)^2)\partial_{YY}^L\Omega_0] dXdY\right|\\
\lesssim & \nu\|\partial_Y\Omega_0\|^2_{H^2}\|1-(\bar U'+\partial_y\psi)^2\|_{L^2}   +\nu \|\partial_Y\Omega_0\|^2_{H^2}\|(\bar U'+\partial_y\psi)\partial_Y(\bar U'+\partial_y\psi)\|_{L^2} \|\Omega_0\|_{L^2}       \\
 & + \nu\|\partial_Y\Omega_0\|^2_{H^{\sigma}}\|1-(\bar U'+\partial_y\psi)^2\|_{H^2}     +   \nu \|\partial_Y\Omega_0\|^2_{H^{\sigma}}\|(\bar U'+\partial_y\psi)\partial_Y(\bar U'+\partial_y\psi)\|_{H^{\sigma-2}} \|\Omega_0\|_{H^4}  \\
\lesssim & \nu (1+\|1+(\bar U'+\partial_y\psi)\|_{H^2})\Big(\|\partial_Y A^{\sigma}\Omega_0\|^2_2 \|1-(\bar U'+\partial_y\psi)\|_{H^2} \\
& +\|\partial_Y(\bar U'+\partial_y\psi)\|_{H^{\sigma-2}}\|\partial_Y \Omega_0\|^2_{H^{\sigma}}\|\Omega_0\|_{H^4}\Big)\\
\lesssim & \epsilon \nu\nu^{1/3} \|\partial_YA^{\sigma}\Omega_0\|^2_{2} + \epsilon \nu\nu^{1/3} (\|\bar U''\|^2_{H^{\sigma-2}} +\|\partial_Yh\|^2_{H^{\sigma-2}}).
\end{split}
\end{equation}
Here we have used
$$\partial_Y\bar U'=(\bar U'+\psi_y)^{-1}\bar U'',\ \ \ \partial_Y\partial_y\psi=\partial_Yh$$ 
and the estimate
\begin{equation}
\begin{split}
\|\partial_Y(\bar U'+\partial_y\psi)\|_{H^{\sigma-2}}  \lesssim & (1+\|\bar U'+\partial_y\psi\|_{H^{\sigma-2}})\|\bar U''\|_{H^{\sigma-2}} +\|\partial_Yh\|_{H^{\sigma-2}}.
\end{split}
\end{equation}

\underline{Estimate of the non-zero mode.}
We use the paraproduct decomposition in $Y$, to obtain (see \eqref{e2.45})
$$E^{\neq}=E^{\neq}_{LH}+E^{\neq}_{HL}+E^{\neq}_{HH},$$
where
\begin{equation}
\begin{split}
E_{LH}^{\neq}=  & -\sum_{M\geq8} \nu\int A^{\sigma}\Omega_{\neq} A^{\sigma}\left[(1-(\bar U'+\partial_y\psi)^2)_{<M/8}(\partial_{Y}-t\partial_X)^2(\Omega_{\neq})_{M}\right]dXdY,\\
E_{HL}^{\neq}=  & -\sum_{M\geq8} \nu\int A^{\sigma}\Omega_{\neq} A^{\sigma}\left[(1-(\bar U'+\partial_y\psi)^2)_{M}(\partial_{Y}-t\partial_X)^2(\Omega_{\neq})_{<M/8}\right]dXdY,\\
E_{HH}^{\neq}=  & -\sum_{M\in \Bbb D}\sum_{\frac18M\leq M'\leq 8M} \nu\int A^{\sigma}\Omega_{\neq} A^{\sigma}\left[(1-(\bar U'+\partial_y\psi)^2)_{M}(\partial_{Y}-t\partial_X)^2(\Omega_{\neq})_{M'}\right]dXdY.
\end{split}
\end{equation}

In the following, we estimate the three terms. 

\underline{LH term.} For this term, we have
\begin{equation*}
\begin{split}
|E_{LH}^{\neq}| \lesssim   & \nu\sum_{M\geq8}\sum_{k\neq 0}\int_{\eta,\xi} A^{\sigma}_k|\overline{\widehat{\Omega_{k}}}(\eta)| A^{\sigma}_k\left[\widehat{(1-(\bar U'+\partial_y\psi)^2)}_{<M/8}(\partial_{Y}-t\partial_X)^2(\Omega_{\neq})_{M}\right]d\xi d\eta\\
\lesssim &  \nu\sum_{M\geq8}\sum_{k\neq 0}\int_{\eta,\xi} \sqrt{k^2+|\eta-kt|^2}A^{\sigma}_k|\overline{\widehat{\Omega_{k}}}(\eta)|   \widehat{\la\partial_Y\ra(1-(\bar U'+\partial_y\psi)^2)}_{<M/8}|\xi-kt|A^{\sigma}_k(\xi)\Omega_{k}(\xi)_{M} d\xi d\eta\\
\lesssim & \nu\sum_{M\geq8}\|(\sqrt{-\Delta_L}A^{\sigma}\Omega_{\neq})_{\sim M}\|_{L^2}^2    \|(1-(\bar U'+\partial_y\psi)^2)\|_{H^4}\\
\lesssim  & \nu \|\sqrt{-\Delta_L}A^{\sigma}\Omega_{\neq}\|^2_2 \|1-(\bar U'+\partial_y\psi)^2\|_{H^4},
\end{split}
\end{equation*}
where we have used the fact that in this regime,
$$\xi\sim\eta\sim M,\ \ \ |k,\eta|\sim|k,\xi|,\ \ \ |\xi-kt|\lesssim |\xi-\eta|+|\eta-kt|\lesssim \la\xi-\eta\ra\sqrt{k^2+|\eta-kt|^2}.$$

\underline{HL term.} For the High-Low term, we have
\begin{equation}
\begin{split}
|E_{HL}^{\neq}| \lesssim   & \nu\sum_{M\geq8}\sum_{k\neq 0}\int_{\eta,\xi} \Big(1_{|\eta|\leq 16|k|} + 1_{|\eta|\geq 16|k|}\Big)A^{\sigma}_k(\eta)|\overline{\widehat{\Omega_{k}}}(\eta)| A^{\sigma}_k(\eta)\\
& \ \ \ \ \ \ \  \ \ \ \ \ \ \  \ \ \ \ \ \times \left[\widehat{(1-(\bar U'+\partial_y\psi)^2)}(\eta-\xi)_{M/8}(\xi-kt)^2\Omega_{\neq}(\xi)_{<M/8}\right]d\xi d\eta.
\end{split}
\end{equation}
In the $|\eta|\leq 16|k|$ regime, we have $|k,\eta|\sim |k| \sim |k,\xi|$ and $|\xi-kt|\lesssim \la \xi-\eta\ra\sqrt{k^2+|\eta-kt}$, while in the $|\eta|\geq16|k|$ regime, we have $|k,\eta|\sim |\eta| \sim |\eta-\xi|\sim M$. Then we have
\begin{equation}
\begin{split}
|E_{HL}^{\neq}| \lesssim   & \nu\sum_{M\geq8}M^{-2} \|(1-(\bar U'+\partial_y\psi)^2)_{\sim M}\|_{H^5} \|\sqrt{-\Delta_L}A^{\sigma}\Omega_{\neq}\|_{L^2}^2 \\
&  + \nu\sum_{M\geq8}\|\Omega_{\neq,\sim M}\|_{H^{\sigma}} \|\partial_Y(1-(\bar U'+\partial_y\psi)^2)\|_{H^{\sigma-1}} \la t\ra^2 \|\Omega_{\neq}\|_{H^5}.
\end{split}
\end{equation}
Summing in $M\geq8$, we  have
$$|E^{\neq}_{HL}| \lesssim \nu \|\sqrt{-\Delta_L}A^{\sigma}\Omega_{\neq}\|^2_2 \|1-(\bar U'+\partial_y\psi)^2\|_{H^5} +\nu \|\Omega_{\neq}\|_{H^{\sigma}} \|\partial_Y(\bar U'+\partial_y\psi)^2\|_{H^{\sigma-1}}\la t\ra^2\|\Omega_{\neq}\|_{H^5}.$$

\underline{HH term.} In this case, it holds $|\eta-\xi|\sim |\xi| \sim M'$, and we can divide 
\begin{equation}
\begin{split}
|E_{HH}^{\neq}| \lesssim   & \nu\sum_{M\in\Bbb D}\sum_{\frac18M\leq M'\leq8M}\sum_{k\neq 0}\int_{\eta,\xi} \Big(1_{|k|\geq 16|\xi|}  +  1_{|k|\leq 16|\xi|}\Big)A^{\sigma}_k(\eta)|\overline{\widehat{\Omega_{k}}}(\eta)| \\
& \ \ \ \ \ \ \ \ \ \ \ \ \ \ \ \ \times A^{\sigma}_k(\eta) \widehat{(1-(\bar U'+\partial_y\psi)^2)}(\eta-\xi)_{M'}|\xi-kt|^2\Omega_{\neq}(\xi)_{M} d\xi d\eta.
\end{split}
\end{equation}
Similarly, we have
\begin{equation}
\begin{split}
|k| \sim |k,\eta| \sim |k|+|\eta-\xi|+|\xi| \sim |k|,  &  \ \ \ \ \ \ \ |k|\geq 16|\xi|,\\
|k,\eta| \lesssim |k| +|\eta| \lesssim  |k|+|\eta-\xi|+|\xi| \sim |k|+|\xi| \lesssim |\xi| \sim |\xi-\eta|,   &  \ \ \ \ \ \ \ |k|\leq 16|\xi|,
\end{split}
\end{equation}
and in both regimes, we have
 $$|\xi-kt|\lesssim |\xi-\eta|+|\eta-kt|\lesssim  \la \xi-\eta\ra\sqrt{k^2+|\eta-kt}.$$
Therefore, we have 
\begin{equation}
\begin{split}
|E_{HH}^{\neq}| \lesssim   & \nu\sum_{M\in\Bbb D}  \|(1-(\bar U'+\partial_y\psi)^2)_{\sim M}\|_{H^3} \|\sqrt{-\Delta_L}A^{\sigma}\Omega_{\neq}\|_{L^2}\|\sqrt{-\Delta_L}A^{\sigma}\Omega_{\neq,M}\|_{L^2} \\
\lesssim & \nu  \|(1-(\bar U'+\partial_y\psi)^2\|_{H^3} \|\sqrt{-\Delta_L}A^{\sigma}\Omega_{\neq}\|_{L^2}^2.
\end{split}
\end{equation}

Putting these estimates together, we obtain
\begin{equation}\label{e5.2}
\begin{split}
E^{\neq} \lesssim &  \nu \|\sqrt{-\Delta_L}A^{\sigma}\Omega_{\neq}\|^2_2 \|1-(\bar U'+\partial_y\psi)^2\|_{H^5} +\nu \|\Omega_{\neq}\|_{H^{\sigma}} \|\partial_Y(\bar U'+\partial_y\psi)^2\|_{H^{\sigma-1}}\la t\ra^2\|\Omega_{\neq}\|_{H^5}\\
\lesssim  &  \nu (1+(\|\bar U'\|_{H^5}+\|\partial_y\psi\|_{H^5}))\|\sqrt{-\Delta_L}A^{\sigma}\Omega_{\neq}\|^2_2 (\|\bar U'\|_{H^5}+\|\partial_y\psi\|_{H^5}) \\
&  +  \nu\left( (1+\|\bar U'+\partial_y\psi\|_{H^{\sigma-1}})\|\bar U''\|_{H^{\sigma-1}} +\|\partial_Yh\|_{H^{\sigma-1}}\right) \|\Omega_{\neq}\|_{H^{\sigma}}   \la t\ra^2\|\Omega_{\neq}\|_{H^5}  \\
\lesssim   &  \nu\epsilon\nu^{1/6} \|\sqrt{-\Delta_L}A^{\sigma}\Omega_{\neq}\|^2_2    + (\epsilon\nu^{1/3})^2\nu\frac{\la t\ra^2}{\la \nu t^3\ra} (\|\bar U''\|_{H^{\sigma-1}}+\|\partial_Yh\|_{H^{\sigma-1}}).
\end{split}
\end{equation}

Integrating \eqref{e2.45} over $[1,t]$ in time gives
\begin{equation}
\begin{split}
\int_1^t  &  DEd\tau =   \nu\int_1^t\int A^{\sigma}\Omega A^{\sigma}(\tilde\Delta_t\Omega)dXdYd\tau\\
\leq &  -\int_1^t\nu\|\sqrt{-\Delta_L}A^{\sigma}\Omega\|^2_2 d\tau  +  \int_1^t |E^0| +|E^{\neq}|d\tau\\
\leq  &  -\int_1^t\nu\|\sqrt{-\Delta_L}A^{\sigma}\Omega\|^2_2 d\tau  + C\epsilon\nu^{1/3}\int_1^t\nu (\|\bar U''\|^2_{H^{\sigma-2}} +\|\partial_Yh\|^2_{H^{\sigma-2}})d\tau     \\
&  +  C\epsilon\nu^{1/6}  \int_1^t \nu \|\sqrt{-\Delta_L}A^{\sigma}\Omega\|^2_2d\tau   + C(\epsilon\nu^{1/3})^2\int_1^t\nu\frac{\la \tau\ra^2}{\la \nu \tau^3\ra} (\|\bar U''\|_{H^{\sigma-1}}+\|\partial_Yh\|_{H^{\sigma-1}})d\tau. 
\end{split}
\end{equation}
Taking $\epsilon$ small  enough and using the decay estimate of $\bar U''$ and the bootstrap hypotheses, we have
\begin{equation}
\begin{split}
\int_1^t    DEd\tau \leq  &   -\frac78\int_1^t\nu\|\sqrt{-\Delta_L}A^{\sigma}\Omega\|^2_2 d\tau  + C\epsilon\nu^{1/3}\int_1^t\nu (\|\bar U''\|^2_{H^{\sigma-2}} +\|\partial_Yh\|^2_{H^{\sigma-2}})d\tau     \\
& + C(\epsilon\nu^{1/3})^2\int_1^t\nu\frac{\la \tau\ra^2}{\la \nu \tau^3\ra} (\|\bar U''\|_{H^{\sigma-1}}+\|\partial_Yh\|_{H^{\sigma-1}})d\tau \\
\leq  & -\frac78\int_1^t\nu\|\sqrt{-\Delta_L}A^{\sigma}\Omega\|^2_2 d\tau   +C\epsilon\nu^{4/3}\int_1^t\|\bar U''\|^2_{H^{\sigma-2}}d\tau  +C\epsilon^3\nu  \\
&   +  C\epsilon^2\nu  \left(\int_1^t\frac{1}{\la \nu^{1/3} \tau\ra}d\tau\right)^{1/2}  \left(\int_1^t (\|\bar U''\|^2_{H^{\sigma-1}}+\|\partial_Yh\|^2_{H^{\sigma-1}})d\tau \right)^{1/2}    \\
\leq & -\frac78\int_1^t\nu\|\sqrt{-\Delta_L}A^{\sigma}\Omega\|^2_2 d\tau   +C\epsilon^2\nu^{4/3}  +C\epsilon^3\nu    +C\epsilon^3\nu^{2/3}.  
\end{split}
\end{equation}

In summary, we have proved the following Proposition. %(see Proposition 2.5 in \cite{MZ19})
\begin{prop}
Under the bootstrap hypotheses,
\begin{equation}
\begin{split}
\nu\int_1^t\int A^{\sigma}\Omega A^{\sigma}(\tilde\Delta_t\Omega)dXdYd\tau \leq -\frac78\int_1^t\nu\|\sqrt{-\Delta_L}A^{\sigma} \Omega\|^2_2d\tau +C\max\{\nu^{2/3},\epsilon\nu{1/2},\epsilon\}\cdot \epsilon^2\nu^{2/3}.
\end{split}
\end{equation}
\end{prop}

\section{Estimate of the convective term}
\setcounter{equation}{0}

In this section, we estimate \eqref{e2.44}.
 
\underline{Transport term $T_N$.}
Now, we consider the transport term $T_N$ in \eqref{e2.44-1}. Recall
\begin{equation*}
\begin{split}
T_N= &  \int A^{\sigma}\Omega \left[A^{\sigma}(u_{<N/8}\cdot\nabla_{X,Y}\Omega_N -u_{<N/8}\cdot\nabla_{X,Y}A^{\sigma}\Omega_N)\right] dXdY,
\end{split}
\end{equation*}
where
$$u(t,X,Y)=\left(
      \begin{array}{c}
         0\\
         g\\
      \end{array}
 \right) +(\bar U'+\partial_y\psi) \nabla^{\bot}_{X,Y}\phi_{\ne}.$$
In the Fourier space, we can rewrite
\begin{equation}
\begin{split}
T_N= & i \sum_{k,l} \int_{\eta,\xi} A^{\sigma}_k(\eta)\overline{\widehat{\Omega _k}}(\eta)   \hat u_{k-l}(\eta-\xi)_{<N/8}  \cdot   (l,\xi)   A^{\sigma}_l(\xi)   \widehat{\Omega_l}(\xi)_{N} \left(\frac{A^{\sigma}_{k}(\eta)}{A^{\sigma}_{l}(\xi)}-1\right)\left( \frac{w_l(t,\xi)}{w_{k}(t,\eta)}\right)   d\xi d\eta \\
&  +   i \sum_{k,l} \int_{\eta,\xi} A^{\sigma}_k(\eta)\overline{\widehat{\Omega _k}}(\eta)   \hat u_{k-l}(\eta-\xi)_{<N/8}  \cdot   (l,\xi)   A^{\sigma}_l(\xi)   \widehat{\Omega_l}(\xi)_{N} \left( \frac{w_l(t,\xi)}{w_{k}(t,\eta)}-1\right)   d\xi d\eta,
\end{split}
\end{equation}
thanks to the definition of the multipliers
\begin{equation}\label{e3.9}
\begin{split}
A_k^{\sigma}(t,\eta)=\frac{\la k,\eta\ra^{\sigma}}{w_k(t,\eta)}.
\end{split}
\end{equation}
As in \cite{MZ19}, we have
\begin{equation}
\begin{split}
|T_N^1|   \lesssim &  \|A^{\sigma}\Omega_{\sim N}\|_2   \|A^{\sigma}\Omega\|_2  \|u\|_{H^4},\\
|T_N^1|   \lesssim &  \|A^{\sigma}\Omega_{\sim N}\|_2   \|A^{\sigma}\Omega_N\|_2  (\|g\|_{H^4}+\|u_{\neq}\|_{H^4}) \left(\nu^{-1/3}\chi_{t\lesssim \nu^{-1/3}}(t)  +\nu^{\beta/3}t^{1-\beta}\chi_{t\gtrsim \nu^{-1/3}}(t)\right).
\end{split}
\end{equation}
Now, since 
$$u_{\neq} =(\bar U'+\partial_y\psi) \nabla^{\bot}_{X,Y}\phi_{\ne},  $$
we have from Lemma \ref{lem3.1} with $s'=2$ that
$$\|u_{\neq}\|_{H^4}\lesssim  (\|\bar U'\|_{H^4}+\|h\|_{H^4})\|\nabla^{\bot}_{X,Y}\phi_{\ne}\|_{H^4} \lesssim  (\|\bar U'\|_{H^4}+\|h\|_{H^4})\frac{1}{\langle t\rangle^{2}}\|\langle\partial_X\rangle^{2}\Omega_{\ne}\|_{H^{7}}.$$
Therefore, by using the bootstrap hypotheses of $g$, $h$ and $\Omega$, and  the assumption of smallness of $\bar U$, we have
$$|T_N|  \lesssim  \epsilon  \|A^{\sigma}\Omega_{\sim N}\|_2   \|A^{\sigma}\Omega_N\|_2  \left( \frac{\chi_{t\lesssim \nu^{-1/3}}(t)}{\la t\ra^2}  + \frac{\nu^{(\beta+1)/3}t^{1-\beta}\chi_{t\gtrsim \nu^{-1/3}}(t)}{\la t\ra^{1+\beta}}\right). $$

By  integration in $t$ over $[1,t]$, we have the following estimate  
\begin{equation}\label{prop2.6}
\begin{split}
\int_1^t\sum_{N\geq8}T_N(\tau)d\tau \lesssim \epsilon \sup_{\tau\in [1,t]}\|A^{\sigma}\Omega(\tau)\|^2_2.
\end{split}
\end{equation}

\underline{Remainder term $\mathcal R$.}
In the regime $N\sim N'$, we have $|l,\xi|\sim |k-l,\eta-\xi|$ and hence $A^{\sigma}_k(\eta)\lesssim \la k,\eta\ra^{\sigma} \lesssim \la l,\xi\ra^{\sigma} +\la k-l,\eta-\xi\ra^{\sigma} \sim  \la k-l,\eta-\xi\ra^{\sigma}   \lesssim A^{\sigma}_{k-l}(\eta-\xi)$. Therefore,
$$|\mathcal R|  \lesssim  \sum_{N\in \Bbb D}\|A^{\sigma}\Omega\|_2 \|\Omega_{\sim N}\|_{H^{\sigma}} \|u_{N}\|_{H^3} \lesssim \|A^{\sigma}\Omega\|^2 \|u\|_{H^3} \lesssim  \frac{\epsilon \nu^{1/3}}{\la t\ra^2}\|A^{\sigma}\Omega\|^2.$$
Integrating over $[1,t]$ in time, we have
\begin{equation}\label{prop2.7}
\begin{split}
\int_1^t|\mathcal R(\tau)| d\tau \lesssim C\epsilon\sup_{\tau\in [1,t]}\|A^{\sigma}\Omega(\tau)\|^2_2.
\end{split}
\end{equation}

\underline{Reaction term $R_N$.} Now we estimate the reaction term $R_N$.  From \eqref{e2.18}, we have
$$u(t,X,Y)=\left(
      \begin{array}{c}
         0\\
         g\\
      \end{array}
 \right)   + \nabla^{\bot}_{X,Y}\phi_{\ne} + (\bar U'-1+\partial_y\psi) \nabla^{\bot}_{X,Y}\phi_{\ne}.$$
Accordingly, $R_N$ is divided into four parts
\begin{equation}
\begin{split}
R_N=    &   \int A^{\sigma}\Omega A^{\sigma}((\nabla^{\bot}_{X,Y}\phi_{\ne})_{N} \cdot\nabla_{X,Y}\Omega_{<N/8})dXdY  + \int A^{\sigma}\Omega  A^{\sigma}(g_N\partial_{Y}\Omega_{<N/8}) dXdY \\
&  +     \int A^{\sigma} \Omega A^{\sigma}  (((\bar U'-1+\partial_y\psi) \nabla^{\bot}_{X,Y}\phi_{\ne})_{N}\cdot\nabla_{X,Y}A^{\sigma}\Omega_{<N/8}) dXdY  \\
&  -    \int A^{\sigma}\Omega  u_N\cdot\nabla_{X,Y}A^{\sigma}\Omega_{<N/8} dXdY    =  R_N^{1} +   R_N^{2}  +   R_N^{1,\epsilon}  + R_N^{3}.
\end{split}
\end{equation}
The first term, after long and tedious estimate (see \cite{MZ19} for more details), we have
\begin{equation}
\begin{split}
|R_N^1| \lesssim &  \left(\frac{\epsilon\nu^{1/3}}{\la \nu^{1/2}t^{3/2}\ra}   +\frac{\epsilon\nu^{1/3}}{\la \nu t^{3}\ra}  +\frac{\epsilon\nu^{1/3}}{\la t^2\ra }\right)  \|A^{\sigma}\Omega_{\sim N}\|_{L^2}  \|A^{\sigma}\Delta_L\Delta_t^{-1}\Omega_{\ne, N}\|_{L^2}   \\
&  + \epsilon \|\sqrt{\frac{\partial_t w}{w}}A^{\sigma}\Omega_{\sim N}\|_{L^2}    \|\sqrt{\frac{\partial_t w}{w}}\chi_RA^{\sigma}\Delta_L\Delta_t^{-1}\Omega_{\ne, N}\|_{L^2}.
\end{split}
\end{equation}
For the $R_N^2$ term, since $|k,\eta-\xi|\lesssim \frac{3}{16}N\leq \frac38|\xi|\sim |k,\eta| $, we have
$$|R_N^2|\lesssim \|A^{\sigma}\Omega_{\sim N}\|_{L^2} \|g_N\|_{H^{\sigma}}  \|f\|_{H^3}.$$
For the $R_N^3$ term, since $|k-l,\eta-\xi|\lesssim \frac{3}{16}N\leq \frac38|l,\xi|$ and $A^{\sigma}$ lands on the low frequencies, we have
$$|R_N^3|\lesssim \|A^{\sigma}\Omega_{\sim N}\|_{L^2} \|u_N\|_{H^{3}}  \|f\|_{H^{\sigma}}.$$
For the $R_N^{1,\epsilon}$ term, we have
\begin{equation}
\begin{split}
|R_N^{1,\epsilon}| \lesssim &  \epsilon\nu^{1/3}\left(\frac{\epsilon\nu^{1/3}}{\la \nu^{1/2}t^{3/2}\ra}   +\frac{\epsilon\nu^{1/3}}{\la \nu t^{3}\ra}  +\frac{\epsilon\nu^{1/3}}{\la t^2\ra }\right)  \|A^{\sigma}\Omega_{\sim N}\|_{L^2}  \|A^{\sigma}\Delta_L\Delta_t^{-1}\Omega_{\ne, N}\|_{L^2}   \\
&  + \epsilon\nu^{1/3}\epsilon \|\sqrt{\frac{\partial_t w}{w}}A^{\sigma}\Omega_{\sim N}\|_{L^2}    \|\sqrt{\frac{\partial_t w}{w}}\chi_RA^{\sigma}\Delta_L\Delta_t^{-1}\Omega_{\ne, N}\|_{L^2}\\
&  + \frac{\epsilon^2\nu^{2/3}}{\la t^2\ra }   \|A^{\sigma}\Omega_{\sim N}\|_{L^2}   \Big( \|\partial_Y(\bar U'+h)\|_{H^{\sigma-1}}   + \|\bar U'-1+h\|_{H^{\sigma-1}}      + \|A^{\sigma}\Delta_L\Delta_t^{-1}\Omega_{\ne, N}\|_{L^2} \Big).
\end{split}
\end{equation}
Putting these estimates together, integrating over $[1,t]$ and using Lemmas \ref{lem3.1} and \ref{lem3.3}, we have
\begin{equation}
\begin{split}
\int_1^t\sum_{N\geq 8}|R_N(\tau)|d\tau  \lesssim &  \int_1^t\left(\frac{\epsilon\nu^{1/3}}{\la \nu^{1/2}t^{3/2}\ra}   +\frac{\epsilon\nu^{1/3}}{\la \nu t^{3}\ra}  +\frac{\epsilon\nu^{1/3}}{\la t^2\ra }\right)  \|A^{\sigma}\Omega \|_{L^2}  \|A^{\sigma}\Delta_L\Delta_t^{-1}\Omega_{\ne}\|_{L^2}  d\tau   \\
&  +  \epsilon   \int_1^t  \|\sqrt{\frac{\partial_t w}{w}}A^{\sigma}\Omega \|_{L^2}  \left( \|\sqrt{\frac{\partial_t w}{w}}\Omega_{\ne}\|_{H^{\sigma}}  + \frac{\epsilon^2\nu^{1/2}}{\la \nu t^{3}\ra} \right)  d\tau  \\
&  +\int_1^t\epsilon\nu^{1/3}\int_1^t\|A^{\sigma}\Omega\|_{L^2} \|g\|_{H^{\sigma}}d\tau    \\
&    + \frac{\epsilon^2\nu^{2/3}}{\la t^2\ra }   \|A^{\sigma}\Omega\|_{L^2}   \Big( \|\partial_Y(\bar U'+h)\|_{H^{\sigma-1}}   + \|\bar U'-1+h\|_{H^{\sigma-1}} \Big)d\tau.\end{split}
\end{equation}
By Young's inequality and bootstrap hypotheses, it gives that
\begin{equation}\label{prop2.8}
\begin{split}
\int_1^t\sum_{N\geq8}R_N(\tau)d\tau \lesssim \epsilon \sup_{\tau\in [1,t]}\|A^{\sigma}\Omega(\tau)\|^2_2 +\epsilon\int_1^tCK_{w}(\tau)d\tau +\epsilon^5\nu^{\frac23}.
\end{split}
\end{equation}
In summary, putting the estimates for $T_N, R_N$ and $\mathcal R$ in \eqref{prop2.6},  \eqref{prop2.7} and \eqref{prop2.8} together, we obtain \eqref{prop2.678}:
\begin{equation*}
\begin{split}
\int_1^tCM(\tau)d\tau \lesssim \epsilon \sup_{\tau\in [1,t]}\|A^{\sigma}\Omega(\tau)\|^2_2 +\epsilon\int_1^tCK_{w}(\tau)d\tau +\epsilon^5\nu^{\frac23}.
\end{split}
\end{equation*}

\section{Estimate of the source term}

In this section, we estimate the source term in \eqref{eq2.1}.

\underline{LH source term.} We write
\begin{equation}
\begin{split}
S^{LH}= & 2\pi\sum_{k,l}\int_{\eta,\xi} A^{\sigma}_k(\eta) \overline{\hat\Omega}_k(\eta)\cdot A^{\sigma}_k(\eta) \widehat{\bar U''}_{k-l}(\eta-\xi)_{<N/8} \widehat{\partial_X\Delta_t^{-1}\Omega_l}(\xi)_{N}d\xi d\eta.
\end{split}
\end{equation}
Note that in this case, $N/2\leq |l,\xi|\leq 3N/2$ and $|k-l,\xi-\eta|\leq 3N/16\leq 3|l,\xi|/8$, and hence $$|k-l|\leq \frac{3|l|}{8},\ \ \ \frac{5|l|}{8}\leq |k|\leq \frac{11|l|}{8}, \ \ \ \frac{5|\xi|}{8}\leq |\eta|\leq \frac{11|\xi|}{8}, \ \ \ \frac{5|l,\xi|}{8}\leq |k,\eta|\leq \frac{11|l,\xi|}{8}.$$ This shows that $|k-l,\xi-\eta|\sim|k,l|$ and hence $$A^{\sigma}_k(\eta)\sim \la k,\eta\ra^{\sigma}\sim \la l,\xi\ra^{\sigma}.$$
Therefore we have
\begin{equation}
\begin{split}
S^{LH}= & 2\pi\sum_{k,l\neq0}\int_{\eta,\xi} A^{\sigma}_k(\eta) \overline{\hat\Omega}_k(\eta)\cdot A^{\sigma}_l(\xi) \widehat{\bar U''}_{k-l}(\eta-\xi)_{<N/8} \widehat{\partial_X\Delta_L^{-1}\Delta_L\Delta_t^{-1}\Omega_l}(\xi)_{N}d\xi d\eta\\
\lesssim & \sum_{N\in \Bbb D} \|A^{\sigma}\Omega\|_{L^2}\|\bar U''\|_{H^3}\|\Delta_L\Delta_t^{-1}\Omega_{\sim N}\|_{H^{\sigma}}\\
\lesssim & \|A^{\sigma}\Omega\|_{L^2}\|\bar U''\|_{H^{\sigma}}\left(\|\Omega_{\neq}\|_{H^{\sigma}}+\frac{\epsilon\nu^{1/3}}{\la t\ra\la \nu t^3\ra}\|\partial_Yh\|_{H^{\sigma}}\right),
\end{split}
\end{equation}
thanks to the fact that $\|\partial_X\Delta_L^{-1}\|\leq1$ for $k\neq 0$ and Lemma \ref{lem3.3}. Similar treatment  can be applied for remainder source term $S^R$ below.

\underline{HL source term.}
We write
\begin{equation*}
\begin{split}
S^{HL}= & 2\pi\sum_{k,l}\int_{\eta,\xi} A^{\sigma}_k(\eta) \overline{\hat\Omega}_k(\eta)\cdot A^{\sigma}_k(\eta) \widehat{\bar U''}_{k-l}(\eta-\xi)_{N} \widehat{\partial_X\Delta_t^{-1}\Omega_{l}}(\xi)_{<N/8}d\xi d\eta.
\end{split}
\end{equation*}
In this regime, we have $|k-l,\eta-\xi|\sim N$ and $|l,\xi|<N/8$,  and hence $|k-l,\eta-\xi|\sim |k,\eta|$. This gives
\begin{equation*}
\begin{split}
S^{HL}= & 2\pi\sum_{k,l}\int_{\eta,\xi} A^{\sigma}_k(\eta) \overline{\hat\Omega}_k(\eta)\cdot A^{\sigma}_k(\eta) \widehat{\bar U''}_{k}(\eta)_{\sim N} \widehat{\partial_X\Delta_t^{-1}\Omega_{l}}(\xi)_{<N/8}d\xi d\eta\\
\lesssim & \sum_{N\in \Bbb D}\|\Omega\|_{H^{\sigma}} \|\bar U''_{\sim N}\|_{H^{\sigma}}\|\Delta_L\Delta_t^{-1}(\Omega_l)_{<N/8}\|_{H^3}\\
\lesssim & \sum_{N\in \Bbb D}\|\Omega\|_{H^{\sigma}} \|\bar U''_{\sim N}\|_{H^{\sigma}}\|\Delta_L\Delta_t^{-1}\Omega\|_{H^3}\\
\lesssim & \|\Omega\|_{H^{\sigma}} \|\bar U''\|_{H^{\sigma}}\left(\|\Omega\|_{H^3}+\frac{\epsilon\nu^{1/3}}{\la t\ra\la \nu t^3\ra}\|\partial_Yh\|_{H^3}\right),
\end{split}
\end{equation*}
thanks to Lemma \ref{lem3.3} in the last inequality. Similar treatments in the remainder source term $S^R$ below can be applied by using the decay of the $ \|\bar U''\|_{H^{\sigma}}$ in Proposition \ref{prop5.1} and the assumption \eqref{L1small}.

\underline{Remainder source term.} On the Fourier side,
\begin{equation}
\begin{split}
S^R= 2\pi\sum_{N\in\Bbb D}\sum_{\frac{N}{8}\leq N'\leq 8N}\sum_{k,l}\int_{\eta,\xi}  A^{\sigma}\overline{\widehat\Omega}_k(\eta) A^{\sigma}_k(\eta)\widehat{\bar U''}_l(\xi)_N\widehat{(\partial_X\Delta_L^{-1} \Delta_L\Delta_t^{-1}\Omega)}_{k-l}(\eta-\xi)_{N'} d\eta d\xi.
\end{split}
\end{equation}
On the support of the integrand, $|l,\xi|\approx |k-l,\eta-\xi|$, thus
$$A_k^{\sigma}\approx \la k,\eta\ra^{\sigma}\lesssim \la l,\xi\ra^{\sigma} +\la k-l,\eta-\xi\ra^{\sigma}\approx \la l,\xi\ra\la k-l,\eta-\xi\ra^{\sigma-1}\approx A_{k-l}^{\sigma}(\eta-\xi),$$
which implies that
$$|S^R|\lesssim \sum_{N\in\Bbb D}\|A^{\sigma}\Omega\|_{2}\|(\bar U'')_N\|_{H^3}\|\Delta_L\Delta_t^{-1}\Omega_{\neq,\sim N}\|_{H^{\sigma}},$$
thanks to $\displaystyle |\partial_X\Delta_L^{-1}|\sim \left|\frac{k}{k^2+|\eta-kt|^2}\right|  \lesssim1$. Therefore, from Lemma \ref{lem3.3} we have
\begin{equation}
\begin{split}
|S^R|\lesssim & \|A^{\sigma}\Omega\|_{2}\|\bar U''\|_{H^3}\left(\|\Omega_{\neq}\|_{H^{\sigma}}+\frac{\epsilon\nu^{1/3}}{\la t\ra\la \nu t^3\ra}\|\partial_Yh\|_{H^{\sigma}}\right)\\
\lesssim & \|\bar U''\|_{H^3}\|A^{\sigma}\Omega\|_{2}^2+ \|\bar U''\|_{H^3}\left(\frac{\epsilon\nu^{1/3}}{\la t\ra\la \nu t^3\ra}\|\partial_Yh\|_{H^{\sigma}}\right)^2 =:S^R_A+S^R_B.
\end{split}
\end{equation}
For the first part, we have
$$S^R_A\lesssim \frac{\epsilon\nu^{5/4}}{\la \nu t\ra^{5/4}}\|A^{\sigma}\Omega\|^2_{2}\lesssim \frac{\epsilon }{\la t\ra^{5/4}}\|A^{\sigma}\Omega\|^2_{2},$$
under the assumption \eqref{L1small}. After integration in time, we have
$$\int_1^tS^R_A\lesssim \int_1^t\frac{\epsilon }{\la \tau\ra^{5/4}}d\tau\cdot\sup_{\tau\in[1,t]} \|A^{\sigma}\Omega(\tau,\cdot)\|^2_{2}\leq C\epsilon\sup_{\tau\in[1,t]} \|A^{\sigma}\Omega(\tau,\cdot)\|^2_{2}.$$
For the second term, we have after integration in time
\begin{equation}
\begin{split}
\int_1^t|S^R_B(\tau)|d\tau\lesssim & \int_1^t\frac{\epsilon\nu^{5/4}}{\la \nu t\ra^{5/4}}\frac{\epsilon^2\nu^{2/3}}{\la t\ra^2\la \nu t^3\ra^2}\|\partial_Yh\|_{H^{\sigma}}^2d\tau\\
\lesssim & \epsilon^3 \nu^{\frac{11}{12}} \int_1^t\nu\|\partial_Yh\|^2_{H^{\sigma}}d\tau\\
\lesssim & \epsilon^3 \nu^{\frac{11}{12}} \cdot 8\epsilon(\epsilon\nu^{1/6})^2 \ \ \ \ \ \ \ (bootstrap\ hypotheses)\\
\lesssim & \epsilon^4\nu^{7/12}(8\epsilon\nu^{1/3})^2.
\end{split}
\end{equation}
In summary, we have
\begin{prop}
Assume \eqref{L1small}, then for any $t\geq1$, it holds
$$\int_1^tS(\tau)d\tau \leq C\epsilon\sup_{\tau\in[1,t]} \|A^{\sigma}\Omega(\tau,\cdot)\|^2_{2}+\epsilon^4\nu^{7/12}(8\epsilon\nu^{1/3})^2.$$
\end{prop}

\section{Coordinate system estimates}\label{Coordinate}
\setcounter{equation}{0}

In this section, we will estimate $g$, $h$ and $\bar h$ in appropriate Sobolev spaces. They satisfy the equations in \eqref{gg}, \eqref{hh} and \eqref{barh}.

\underline{Estimate of $g$ in $H^{\sigma}$.}
Recall that as a function of $(t,Y)$, $g$ evolves according to the following equation
\begin{equation}
\begin{split}
g_t+\frac{2g}t+g\partial_Yg=-\frac{(\bar U'+\partial_y\psi)}t\langle\nabla^{\bot}_{X,Y}\phi_{\ne}\cdot\nabla_{X,Y}V^X\rangle+  \nu\widetilde {\Delta_t}g,
\end{split}
\end{equation}
where $$\widetilde {\Delta_t}g=(\bar U'+\partial_y\psi)^2\partial_{YY}g.$$
We compute the evolution of the norm $\|g\|^2_{H^{\sigma}}$. We have
\begin{equation}\label{e7.1}
\begin{split}
\frac{d}{dt}& \|\la\partial_Y\ra^{\sigma}g\|^2_2= -\frac4t\|\la\partial_Y\ra^{\sigma}g\|^2_2 -2\int \la\partial_Y\ra^{\sigma}g \la\partial_Y\ra^{\sigma}(g\partial_Yg)dY\\
& -\frac2t \int\la\partial_Y\ra^{\sigma}g\la\partial_Y\ra^{\sigma}\left((\bar U'+\partial_y\psi)\langle\nabla^{\bot}_{X,Y}\phi_{\ne}\cdot\nabla_{X,Y}V^X\rangle\right)dY +2\nu\int \la\partial_Y\ra^{\sigma}g\la\partial_Y\ra^{\sigma}(\widetilde{\Delta_t}g)dY\\
= & -\frac4t\|\la\partial_Y\ra^{\sigma}g\|^2_2 +V_1^{H,g} +V_2^{H,g} +V_3^{H,g}.
\end{split}
\end{equation}
For the three terms, it can be estimated that
\begin{equation*}
\begin{split}
|V_1^{H,g}|\lesssim   \left|\int \la\partial_Y\ra^{\sigma}g \la\partial_Y\ra^{\sigma}(g\partial_Yg)dY\right| \lesssim  \|g\|_{H^2}\|\la\partial_Y\ra^{\sigma}g\|^2_2.
\end{split}
\end{equation*}
For the second term, we note that
$$\langle\nabla^{\bot}_{X,Y}\phi_{\ne}\cdot\nabla_{X,Y}V^X\rangle =\langle\nabla^{\bot}_{L}\phi_{\ne}\cdot\nabla_{L}V^X_{\ne}\rangle,$$
which implies that 
\begin{equation*}
\begin{split}
|V_2^{H,g}|\lesssim & \frac1t \|\la\partial_Y\ra^{\sigma}g\|_2\left(\|\bar U'-1\|_{H^{\sigma}}+\|h\|_{H^{\sigma}}\right)\| \langle\nabla^{\bot}_{X,Y}\phi_{\ne}\cdot\nabla_{X,Y}V^X\rangle\|_{H^{\sigma}} \\
 & +\frac1t\|\la\partial_Y\ra^{\sigma}g\|_2 \langle\nabla^{\bot}_{L}\phi_{\ne}\cdot\nabla_{L}V^X\rangle\|_{H^{\sigma}}\\
\lesssim & \frac1t \|\la\partial_Y\ra^{\sigma}g\|_2\bigg(\|\nabla^{\bot}_{L}\phi_{\ne}\|_{H^{\sigma}} \|\nabla_{L}V^X_{\ne}\|_{H^{2}} + \|\nabla^{\bot}_{L}\phi_{\ne}\|_{H^{2}} \|\nabla_{L}V^X_{\ne}\|_{H^{\sigma}}\bigg),
\end{split}
\end{equation*}
thanks to the bootstrap assumption and the assumption \eqref{BS2.2} that $\|\bar U'-1\|_{H^{\sigma}}\lesssim \delta$ in \eqref{e25}. 

For the third term, we have by integration by parts
\begin{equation*}
\begin{split}
V_3^{H,g}= & 2\nu\int \la\partial_Y\ra^{\sigma}g\la\partial_Y\ra^{\sigma}(\partial_Y^2g)dY +2\nu\int \la\partial_Y\ra^{\sigma}g\la\partial_Y\ra^{\sigma}\left(\left(\left(\bar U'+\partial_y\psi\right)^2-1\right)\partial_{YY}g\right)\\
\lesssim & -2\nu\|\partial_Y\la\partial_Y\ra^{\sigma}g\|_2^2 + 2\nu\int \la\partial_Y\ra^{\sigma}g\la\partial_Y\ra^{\sigma}\left(\left((\bar U'+1)(\bar U'-1)+(h+2\bar U')h\right)\partial_{YY}g\right)\\
\lesssim & -2\nu\|\partial_Y\la\partial_Y\ra^{\sigma}g\|_2^2 + \nu\left(1+\|h\|_{H^2}+\|\bar U'\|_{H^2}\right) \Big(\|\partial_Yg\|^2_{H^{\sigma}}(\|h\|_{H^2}+\|\bar U'-1\|_{H^2}) \\
&  +(\|h\|_{H^{\sigma-1}}+\|\bar U'-1\|_{H^{\sigma-1}})\|\partial_Yg\|^2_{H^{\sigma}}  +(\|\partial_Yh\|_{H^{\sigma-2}}+\|\partial_Y\bar U'\|_{H^{\sigma-2}})\|g\|^2_{H^{4}}\Big)\\
\lesssim  & -\nu\|\partial_Y\la\partial_Y\ra^{\sigma}g\|_2^2 + \nu\left(1+\|h\|_{H^2}+\|\bar U'\|_{H^2}\right) \Big(\left(\|\partial_Yh\|_{H^{\sigma-2}}+\|\partial_Y\bar U'\|_{H^{\sigma-2}}\right)\|g\|^2_{H^{4}}\Big),
\end{split}
\end{equation*}
thanks to $\|\bar U'-1\|_{H^{\sigma}}\lesssim \delta$ (see \eqref{e25}). Integrating \eqref{e7.1}, and using Lemma \ref{lem3.3}, we have 
\begin{equation*}
\begin{split}
\sup_{\tau\in[1,t]}\bigg(\tau\|g(\tau)\|_{H^{\sigma}}\bigg) +\int_1^t\|g(\tau)\|_{H^{\sigma}}d\tau \leq \|g(1)\|_{H^{\sigma}} +C\left(\epsilon\nu^{1/3}\sup_{\tau\in[1,t]}\tau\|g(\tau)\|_{H^{\sigma}} +\epsilon^2\nu^{1/3}\right).
\end{split}
\end{equation*}
By taking $\epsilon$ small enough, we get 
\begin{equation*}
\begin{split}
\la t\ra\|g\|_{H^{\sigma}}  +\int_1^t\|g(\tau)\|_{H^{\sigma}}d\tau \leq 2\|g(1)\|_{H^{\sigma}} +C\epsilon^2\nu^{1/3}.
\end{split}
\end{equation*}

\underline{Estimate of $h$ and $\bar h$ in $H^{\sigma-1}$.}
Note that
\begin{equation}%\label{barh}
\begin{split}
\bar h_t+\frac2t\bar h+g\partial_Y\bar h=\frac{(\bar U'+\partial_y\psi)}t\langle\nabla^{\bot}_{X,Y}\phi_{\ne}\cdot\nabla_{X,Y}\Omega\rangle+  \nu\widetilde {\Delta_t}\bar h.
\end{split}
\end{equation}
By taking inner product with $\bar h$ in $H^{\sigma-1}$, we obtain
\begin{equation*}
\begin{split}
\frac12\frac{d}{dt}\|\bar h\|_{H^{\sigma-1}}^2=   & -\frac2t\|\bar h\|_{H^{\sigma-1}}^2 - \int\la\partial_Y\ra^{\sigma-1}\bar h \la\partial_Y\ra^{\sigma-1}(g\partial_Y\bar h)dY \\
& + \frac1t \int \la\partial_Y\ra^{\sigma-1}\bar h \la\partial_Y\ra^{\sigma-1}\Big({(\bar U'+\partial_y\psi)}\langle\nabla^{\bot}_{X,Y}\phi_{\ne}\cdot\nabla_{X,Y}\Omega\rangle\Big)dY\\
& +  \nu\int \la\partial_Y\ra^{\sigma-1}\bar h \la\partial_Y\ra^{\sigma-1} \Big(\big((\bar U'+\partial_y\psi)^2-1\big)\partial_{YY}\bar h\Big)dY     -\nu\|\partial_{Y}\bar h\|_{H^{\sigma-1}}^2\\
= & -\frac2t\|\bar h\|_{H^{\sigma-1}}^2  -\nu\|\partial_{Y}\bar h\|_{H^{\sigma-1}}^2 +\sum_{i=1}^3V_i^{H,\bar h},
\end{split}
\end{equation*}
where $$\widetilde {\Delta_t}F=(\bar U'+\partial_y\psi)^2\partial_{YY}F.$$

For the term $V_{3,\epsilon}^{H,\bar h}$, we estimate
\begin{equation*}
\begin{split}
|V_{3,\epsilon}^{H,\bar h}|\lesssim & \nu\int_{\xi,\eta}\la\eta\ra^{2(\sigma-1)}|\hat{\bar h}(\eta)|\left|(\widehat{1-(\bar U'+\partial_y\psi)^2}(\eta-\xi))|\xi|^2|\hat{\bar h}(\xi)|\right|d\xi d\eta\\
\lesssim &\nu \|(1-(\bar U'+\partial_y\psi)^2)\|_{H^2}\|\partial_Y\bar h\|_{H^{\sigma-1}}^2 +\|(1-(\bar U'+\partial_y\psi)^2)\|_{H^{\sigma-1}}\|\bar h\|_{H^{\sigma-1}} \|\bar h\|_{H^{4}}\\
\lesssim & \nu (1+\|h\|_{H^2})\left(\|h\|_{H^2}\|\partial_Y\bar h\|_{H^{\sigma-1}}^2 +\|\partial_Yh\|_{H^{\sigma-2}}\|\bar h\|_{H^{\sigma-1}} \|\bar h\|_{H^{4}}\right),
\end{split}
\end{equation*}
where the first part of sum can be bounded by $\frac{\nu}{2}\|\partial_Y\bar h\|_{H^{\sigma-1}}^2$ by the a priori assumption \eqref{BS2.2}. Using these estimates, we thus obtain
\begin{equation*}
\begin{split}
\frac12\frac{d}{dt}\|\bar h\|_{H^{\sigma-1}}^2\leq  & -\frac2t\|\bar h\|_{H^{\sigma-1}}^2  +C\|g\|_{H^{\sigma-1}}\|\bar h\|_{H^{\sigma-1}}^2 \\
& +\frac Ct \|\bar h\|_{H^{\sigma-1}}(1+\|h\|_{H^2})\Big(\|\phi_{\neq}\|_{H^{\sigma}}\|\partial_Xf_{\neq}\|_{H^1} +\|\partial_X\phi_{\neq}\|_{H^{1}}\|f_{\neq}\|_{H^{\sigma}}\Big)\\
& +\frac Ct \|\bar h\|_{H^{\sigma-1}}\|h\|_{H^{\sigma-1}}\|\partial_X\phi_{\neq}\|_{H^{2}}\|f_{\neq}\|_{H^{2}}  +C\nu \|\partial_Yh\|_{H^{\sigma-2}}\|\bar h\|_{H^{\sigma-1}} \|\bar h\|_{H^{4}},
\end{split}
\end{equation*}
and hence
\begin{equation*}
\begin{split}
\frac12\frac{d}{dt}(t\|\bar h\|_{H^{\sigma-1}}) \leq  & -\|\bar h\|_{H^{\sigma-1}} \leq  +Ct\|g\|_{H^{\sigma-1}}\|\bar h\|_{H^{\sigma-1}} \\
& +C(\|\phi_{\neq}\|_{H^{\sigma}}\|\partial_Xf_{\neq}\|_{H^1} +\|\partial_X\phi_{\neq}\|_{H^{2}}\|f_{\neq}\|_{H^{\sigma}})  +Ct\nu \|\partial_Yh\|_{H^{\sigma-2}}\|\bar h\|_{H^{4}}.
\end{split}
\end{equation*}
This, thanks to the a priori assumption, implies that
\begin{equation*}
\begin{split}
\sup_{t\in[1,T]}(t\|\bar h\|_{H^{\sigma-1}}) +\int_1^T\|\bar h(\tau)\|_{H^{\sigma-1}}d\tau  \leq \|\bar h(1)\|_{H^{\sigma-1}} +C\left(\epsilon\nu^{1/3}\sup_{t\in[1,T]}(t\|\bar h\|_{H^{\sigma-1}})  +\epsilon^2\nu^{1/3} +\epsilon^2\nu^{7/6}\right).
\end{split}
\end{equation*}
Similarly, we have
\begin{equation*}
\begin{split}
\sup_{t\in[1,T]}\|h(t)\|_{H^{\sigma-1}}^2 +\nu\int_1^T\|h(\tau)\|_{H^{\sigma-1}}^2d\tau  \leq \|h(1)\|_{H^{\sigma-1}}^2 +4\|\bar h\|_{L^1_TH^{\sigma-1}}^2  +\frac38\sup_{t\in[1,T]}\|h(t)\|_{H^{\sigma-1}}^2.
\end{split}
\end{equation*}

\underline{Estimate of $h$ and $\bar h$ in $H^{\sigma}$.}
Recall 
\begin{equation}%\label{barh}
\begin{split}
\bar h_t+\frac2t\bar h+g\partial_Y\bar h=\frac{(\bar U'+\partial_y\psi)}t\langle\nabla^{\bot}_{X,Y}\phi_{\ne}\cdot\nabla_{X,Y}\Omega\rangle+  \nu\widetilde {\Delta_t}\bar h.
\end{split}
\end{equation}
Since $\bar U'$ and $\partial_y\psi$ are independent of $k$, we have by applying $A^{\sigma}$ and then taking inner product that
\begin{equation}\label{S9.2}
\begin{split}
\frac12\frac{d}{dt}\|A^{\sigma}\bar h\|_2^2  =   &  - \int\frac{\partial_tw(t,\eta)}{w(t,\eta)}\left| \frac{\la\eta\ra^{\sigma}\widehat{\bar h}(t,\eta)}{w(t,\eta)}\right|^2d\eta    -\frac2t\|A^{\sigma}\bar h\|_2^2 \\
&    -  \int A^{\sigma}\bar h [A^{\sigma}(g\partial_Y\bar h)  - g\partial_YA^{\sigma}\bar h ] dY    + \frac12  \int \partial_Yg|A^{\sigma}\bar h|^2dY    \\
&  + \frac1t  \int A_0^{\sigma}\bar hA_0^{\sigma} \left({\left(1+(\bar U'-1+\partial_y\psi)\right)}\langle\nabla^{\bot}_{X,Y}\phi_{\ne}\cdot\nabla_{X,Y}\Omega\rangle \right) dY \\
&  -\nu\|\partial_YA^{\sigma} \bar h\|_2^2    +\nu\int  A^{\sigma} \bar h  A^{\sigma} \left( ((\bar U'+\partial_y\psi)^2-1) \partial_{YY}\bar h\right).
\end{split}
\end{equation}
By Littlewood-Paley decomposition in $Y$ and recalling \eqref{lem6.1}, we have for the third term on the RHS
\begin{equation}
\begin{split}
\Bigg| \int A^{\sigma}\bar h [A^{\sigma}(g\partial_Y\bar h)  &   - g\partial_YA^{\sigma}\bar h ] dY\Bigg|\lesssim     \|A^{\sigma}\bar h\|_2^2\|g\|_{H^{\sigma}}  \\
&  +\|A^{\sigma}\bar h\|_2\|\bar h\|_{H^{\sigma}}\|g\|_{H^3} \left(\nu^{-1/3}\chi_{t\lesssim \nu^{-1/3}}(t)  +\nu^{\beta/3}t^{1-\beta}\chi_{t\gtrsim \nu^{-1/3}}(t)\right).
 \end{split}
\end{equation}
The main contribution of the fifth term comes from the ``1" part, and we have
\begin{equation}
\begin{split}
\Bigg| \frac1t  \int    &    A_0^{\sigma}\bar hA_0^{\sigma}   \langle\nabla^{\bot}_{X,Y}\phi_{\ne}\cdot\nabla_{X,Y}\Omega\rangle  dY \Bigg|   \lesssim    \frac{1}{t^3}\|\bar h\|_{H^{\sigma}} \|\Omega_{\ne}\|_{H^{\sigma}}  \|\Omega_{\ne}\|_{H^7} \\ 
&  + \frac1t \|\sqrt{\frac{\partial w}{w}}A^{\sigma}\bar h\|_{L^2}  \|\sqrt{\frac{\partial w}{w}}\chi_RA^{\sigma}\Delta_L\Delta_t^{-1}\Omega_{\ne}\|_{L^2}  \|\nu^{-1/3}\la \nu^{1/3}t\ra^{1+\beta}\|_{H^6}  \\
&  + \frac{1}{t^2}\|\bar h\|_{H^{\sigma}} \|\sqrt{-\Delta_L}A^{\sigma}\Omega_{\ne}\|_{L^2}    \|\Omega_{\ne}\|_{H^4}     + \frac{1}{t}\|A^{\sigma}\bar h\|_{L^2} \|A^{\sigma}\Delta_L{\Delta_t^{-1}}\Omega_{\ne}\|_{L^2}    \|\Omega_{\ne}\|_{H^4}.
 \end{split}
\end{equation}
The residual part corresponding to $(\bar U'-1+\partial_y\psi)$ in the fifth term can be controlled similarly with an additional power of $\epsilon$. The last dissipation error term can be controlled
\begin{equation}
\begin{split}
\Bigg|  \nu\int  & A^{\sigma} \bar h  A^{\sigma} \left( ((\bar U'+\partial_y\psi)^2-1) \partial_{YY}\bar h\right)\Bigg|  \\
\lesssim    &    \nu\|\bar h\|_{L^2} \|\bar h\|_{H^2} \|\bar U'+h\|_{H^2}    +\nu \|\bar U'+h\|_{H^3}\|\partial_YA^{\sigma}\bar h\|_{L^2}^2   +\nu \|\bar U'+h\|_{H^{\sigma-1}}\|\partial_Y\bar h\|_{H^{\sigma}}   \|\partial_Y\bar h\|_{H^{3}}  \\
\leq  &     \frac14\|\partial_YA^{\sigma}\bar h\|_{L^2}^2   +C \nu\|\bar h\|_{L^2} \|\bar h\|_{H^2} \|\bar U'+h\|_{H^2},
\end{split}
\end{equation}
thanks to bootstrap hypotheses.  Inserting these  estimates into \eqref{S9.2}, multiplying by $t^3$, and integrating over $[1,t]$, we obtain
\begin{equation}
\begin{split}
t^3\|A^{\sigma}\bar h(t)\|_2^2   &  +\int_1^t\tau^3CK_{w}^{\bar h}d\tau +\frac12\int_1^t \tau^2\|A^{\sigma}\bar h(\tau)\|_2^2d\tau   +\frac{\nu}{2}\int_1^t \tau^3\|\partial_YA^{\sigma}\bar h(\tau)\|_2^2d\tau \\
\leq  &   \|A^{\sigma}\bar h(t)\|_2^2   +\frac1{100}\sup_{\tau\in [1,t]} \tau^3\|A^{\sigma}\bar h(\tau)\|_2^2  +\frac1{100}\int_1^t\tau^3CK_{w}^{\bar h}d\tau   +C\epsilon^4\nu^{1/3}.
\end{split}
\end{equation}
Similar estimate holds for $h$ in $H^{\sigma}$, and we have
\begin{equation}
\begin{split}
\|h(t)\|_{H^{\sigma}}^2   +  {\nu}\int_1^t \|\partial_Yh(\tau)\|_{H^{\sigma}}^2d\tau \leq  \|h(1)\|_{H^{\sigma}}^2      + \frac1{100}\sup_{\tau\in [1,t]} \|h(t)\|_{H^{\sigma}}^2  +C\epsilon^3\nu^{1/3}.
\end{split}
\end{equation}

\underline{Lower energy estimate of $h$ and $\bar h$ in $H^{\sigma-6}$.}
By applying $\la \partial_Y\ra^{\sigma-6}$ to the equation \eqref{gg} and \eqref{barh} and then taking inner product with $t^4\la \partial_Y\ra^{\sigma-6}g$ and  $t^4\la \partial_Y\ra^{\sigma-6}\bar h$, respectively, we can show under bootstrap hypotheses that
\begin{equation}
\begin{split}
\sup_{\tau\in [1,t]}  &  \tau^4\|g(\tau),\bar h(\tau)\|_{H^{\sigma-6}}^2     +\nu\int_1^t  \tau^4\|\partial_Yg(\tau),\partial_Y\bar h\|_{H^{\sigma-6}}^2  d\tau  \\
\leq  &   \|g(1),\bar h(1)\|_{H^{\sigma-6}}^2    +C\epsilon\nu^{1/3}\sup_{\tau\in [1,t]}\tau^4\|g(\tau),\bar h(\tau)\|_{H^{\sigma-6}}^2    + C\epsilon^2\nu^{2/3}\sup_{\tau\in [1,t]}\tau^2\|g(\tau),\bar h(\tau)\|_{H^{\sigma-6}}.
\end{split}
\end{equation}
By taking $\epsilon$ small enough, the first two bootstrap hypotheses \eqref{e2.25} on Lower regularity can be proved.

\section{Decay estimate of vorticity}\label{DecayVorticity}
\setcounter{equation}{0}

\underline{Decay estimate of nonzero mode: Enhanced dissipation.} 
Consider only $t$ such that $\nu t^3\geq1$. Recall that
$$\|A_E^{s}\Omega\|^2_2=\sum_{k\neq0}\int_{\eta} \la k,\eta\ra^{2s}\left|D(t,\eta)\hat\Omega_k(t,\eta)\right|^2d\eta,$$
where $D(t,\eta)=\frac{1}{3}\nu|\eta|^3+ \frac{1}{24}\nu(t^3-8|\eta|^3)_+.$

By direct computation, we have
\begin{equation}\label{e10.1}
\begin{split}
\frac12\frac{d}{dt}\|A_E^{s}\Omega\|^2_2 = & \sum_{k\neq0}\int_{\eta} \frac{\partial_tD(t,\eta)}{D(t,\eta)}\left|A_E^s\hat\Omega_k(t,\eta)\right|^2d\eta  +  \nu\int A_E^{s}\Omega A_E^{s} \widetilde {\Delta_t}\Omega   dYdX  \\ 
&- \int A_E^{s}\Omega A_E^{s}(u(t,X,Y)\cdot\nabla_{X,Y}\Omega)dYdX  + \int A_E^{s}\Omega A_E^{s}  (\bar U''\partial_X\phi)   dYdX=\sum_{i=1}^4 E_i.
\end{split}
\end{equation}
It is noted that by definition of $D(t,\eta)$,  $E_1\leq 3\nu t^2\|1_{t\geq{2\eta}}A_E^s\hat\Omega_k(t,\eta)\|^2_2$. For the dissipation term $E_2$,  similar to that of \cite[\S10]{MZ19}, we have
\begin{equation*}
\begin{split}
E_2 = -\nu\left\| (-\Delta_L)^{1/2} A_E^{s}\Omega \right\|_2^2  \underbrace{- \nu\int A_E^{s}\Omega A_E^{s} \left(\left((\bar U'+\partial_y\psi)^2-1\right)\partial_{YY}^L\right) \Omega   dYdX}_{E^{\nu}}.
\end{split}
\end{equation*}
Thanks to the fact that $|\xi-kt|\leq |\xi-\eta|+|\eta-kt|\leq \la \xi-\eta\ra \sqrt{k^2+|\eta-kt|^2}$, we have for $E^{\nu}$ that
\begin{equation*}
\begin{split}
|E^{\nu}| & \lesssim \nu(1+\|h\|_{H^2}+\|\bar U'-1\|_{H^2})(\|h\|_{H^6}+\|\bar U'-1\|_{H^6}) \left\| \sqrt{-\Delta_L}A^s_E\Omega\right\|_2^2\\
& \lesssim \nu \epsilon\nu^{1/3} \left\| \sqrt{-\Delta_L}A^s_E\Omega\right\|_2^2.
\end{split}
\end{equation*}
Therefore, 
\begin{equation*}
\begin{split}
E_1+E_2\leq  -\frac18\nu\|\sqrt{-\Delta_L}A^s_E\Omega\|_2^2  + C \nu \epsilon\nu^{1/3} \left\| \sqrt{-\Delta_L}A^s_E\Omega\right\|_2^2.
\end{split}
\end{equation*}

For the convective nonlinear term $E_3$, we have 
\begin{equation*}
\begin{split}
E_3 = & -\int A_E^{s}\Omega A_E^{s}(g\partial_{Y}\Omega)dYdX -\int A_E^{s}\Omega A_E^{s}\left(((\bar U'+\partial_y\psi) \nabla^{\bot}_{X,Y}\phi_{\ne})\cdot\nabla\Omega\right)dYdX\\
= & :E_{31}+E_{32}  \ \ \ \ \ \ \ \  (using\ \eqref{e2.18}_3)
\end{split}
\end{equation*}
and by the same arguments, we have 
\begin{equation*}
\begin{split}
E_{31}+E_{32} \lesssim \frac{\epsilon\nu^{1/3}}{\la t\ra^2} (\|A^s_E\Omega\|^2_2 +\|A^{\sigma}\Omega\|_2\|A^s_E\Omega\|_2).
\end{split}
\end{equation*}

For $E_4$, which is an extra term compared to \cite{MZ19}, we have by Young's convolution inequality that
\begin{equation}\label{eE4}
\begin{split}
E_4 = & \int A_E^{s}\Omega A_E^{s}  (\bar U''\partial_X\phi)   dYdX\\
= & \sum_{k\neq 0}\int_{\eta,\xi}  A^s_E(k,\xi)\widehat{\bar U''}(\xi-\eta)\frac{ik}{k^2+|\eta-k t|^2}\widehat{(\Delta_L\Delta_t^{-1}\Omega_{\neq})}(k,\eta) A^s_E(k,\xi)\bar{\hat{\Omega_{k}}}(\xi) d\eta d\xi\\
\lesssim & \sum_{k\neq 0}\int_{\eta,\xi}  \la\xi-\eta\ra^{s+3}\widehat{\bar U''}(\xi-\eta)A^s_E(k,\eta)\widehat{(\Delta_L\Delta_t^{-1}\Omega_{\neq})}(k,\eta) A^s_E(k,\xi)\bar{\hat{\Omega_{k}}}(\xi) d\eta d\xi\\
\lesssim &  \|\bar U''\|_{H^{s+4}}\|A^s_E(k,\eta)\widehat{(\Delta_L\Delta_t^{-1}\Omega_{\neq})}(k,\eta)\|_2\|A^s_E\Omega_{\neq}\|_2.% \\
%&  (details\ are\ needed\ why\ \|A^s_E(k,\eta)\widehat{(\Delta_L\Delta_t^{-1}\Omega_{\neq})}(k,\eta)\|\leq \|A^s_E(k,\eta)\Omega_{\neq}\|!)
\end{split}
\end{equation}
Since 
\begin{equation}\label{e4.1-1}
\Delta_L\phi_{\neq} =\Omega_{\neq} +(1-(\bar U'+\partial_y\phi)^2)\partial_{YY}^L\phi_{\neq}  -(\bar U''+\partial_{yy}\psi)\partial_Y^L\phi_{\neq}.
\end{equation}
Now, we consider $\|A^s_E\left(1-(\bar U'+\partial_y\phi)^2\partial_{YY}^L\phi_{\neq}\right)\|_2$. According to Lemma \ref{lem4.5}, we have
\begin{equation}
\begin{split}
\|A^s_E& ((1-(\bar U' +\partial_y\phi)^2)\partial_{YY}^L\phi_{\neq})\| \\
\lesssim & \nu\|(1-(\bar U'+\partial_y\phi)^2)\partial_{YY}^L\phi_{\neq}\|_{H^{s+3}} +\nu t^3\|(1-(\bar U'+\partial_y\phi)^2)\partial_{YY}^L\phi_{\neq}\|_{H^s}\\
\lesssim & \nu\|(\bar U'-1)+\partial_y\psi\|_{H^{s+3}}\|1+\bar U'+\partial_y\psi\|_{H^{s+3}} \|\Delta_L\phi_{\neq}\|_{H^{s+3}}\\
& +  \|(\bar U'-1)+\partial_y\psi\|_{H^{s}}\|1+\bar U'+\partial_y\psi\|_{H^{s}} (\nu t^3 \|\Delta_L\phi_{\neq}\|_{H^{s}})\\
\lesssim & \epsilon\nu^{7/6} \|\Delta_L\phi_{\neq}\|_{H^{s+3}} +\epsilon\nu^{1/6} \|A^s_E\Delta_L\phi_{\neq}\|_{L^2}\\
\lesssim & \epsilon\nu^{7/6} \|\Omega_{\neq}\|_{H^{s+3}} +\epsilon\nu^{1/6} \|A^s_E\Delta_L\phi_{\neq}\|_{L^2}\\
\lesssim & \epsilon^3\nu^{11/6} +\epsilon\nu^{1/6} \|A^s_E\Delta_L\phi_{\neq}\|_{L^2},
\end{split}
\end{equation}
thanks to Lemma \ref{lem3.1} and the bootstrap assumption \eqref{e2.22} in the last two inequalities. In the same manner, if we apply $A^s_E$ to \eqref{e4.1-1}, and take the $L^2$-norm, we obtain
\begin{equation}
\|A^s_E\Delta_L\phi_{\neq}\|  \leq  \|A^s_E\Omega_{\neq}\| + C\epsilon\nu^{1/6} \|A^s_E\Delta_L\phi_{\neq}\|_{L^2}  +C\epsilon^3\nu^{11/6}.
\end{equation}
Taking $\epsilon$ suitably small,  this gives that 
\begin{equation}
\|A^s_E\Delta_L\phi_{\neq}\|  \leq 2\|A^s_E\Omega_{\neq}\| +C\epsilon^3\nu^{11/6}.
\end{equation}
Integrating \eqref{e10.1} gives that
\begin{equation}
\begin{split}
\|A_E^{s}\Omega\|^2_2& +\frac15\nu\int_1^t\|\sqrt{-\Delta_L}A^s_E\Omega\|_2^2d\tau  \leq  \|A_E^{s}\Omega(1)\|^2_2  \\
 &+ C \epsilon\nu^{1/3} \left\| A^s_E\Omega(t)\right\|_2^2 +\left(\int_1^t \|\bar U''\|_{H^{s+4}}d\tau\right) \|A^s_E\Omega_{\neq}(t)\|_2^2  +C\epsilon^3\nu\\
\leq & \|A_E^{s}\Omega(1)\|^2_2   + C \epsilon  \left\| A^s_E\Omega(t)\right\|_2^2 +C\epsilon^3\nu,
\end{split}
\end{equation}
thanks to the decay of $\|\bar U''\|$ in Proposition \ref{prop5.1} and the assumption that $\|U(y)-y\|_{L^1}\leq \epsilon\nu^{5/4}$ there.

\underline{Decay estimate of zero mode.} 
In this part, we consider the decay of the zero mode of $\Omega$. The zero mode $\Omega_0$ satisfies
$$\partial_t\Omega_0 +g\partial_{Y}\Omega_0 +(\bar U'+\partial_y\psi)\la\nabla^{\bot}_{X,Y}\phi_{\ne}\cdot\nabla_{X,Y}\Omega\ra = \underbrace{\bar U''\la\partial_X\phi\ra}_{=0} +\nu(\bar U'+\partial_y\psi)^2\partial_{YY}\Omega_0,$$
exactly the same equation as the case when the background shear flow is $U(y)=y$. Therefore, by the same estimate as in \cite{MZ19}, we have
\begin{equation}\label{e2.36}
\begin{split}
\sup_{\tau\in [1,t]}  &   \left(\|\Omega_0(\tau)\|^2_{H^{s}} +\frac{\tau\nu}{2}\|\partial_Y\Omega_0\|^2_{H^{s}}\right)     +\nu\int_1^t  \Big(\|\partial_Y\Omega_0(\tau)\|^2_{H^{s}}  +\frac{\tau\nu}{2}\|\partial_Y\Omega_0(\tau)\|^2_{H^{s}}\Big) d\tau\\
 \leq  &  \left(2\|\Omega_0(1)\|^2_{H^{s}} +\nu\|\partial_Y\Omega_0(1)\|^2_{H^{s}}\right) + C\epsilon^3\nu^{2/3}.
\end{split}
\end{equation}

\section{Appendix}
\setcounter{equation}{0}

\subsection{Weights}
In the appendix, we give the definition of the weights and their properties. Before this, define $t_{m,\eta}=\frac{2\eta}{2m+1}$ for $|m|=0,1,2,\cdots$ and $m\eta\geq0$ and $I_{m,\eta}=[t_{m,\eta},t_{m-1,\eta}]$ for $m=1,2,\cdots$ to denote any resonant interval with $\eta\geq(2m+1)m$. For $|\eta|\geq3$, we denote $E(\sqrt{|\eta|})$ the largest integer $m$ that satisfies $(2m+1)m\leq|\eta|$ and then $E(\sqrt{|\eta|})\approx \sqrt{|\eta|}$. Set $t(\eta)=\frac{2\eta}{2E(\sqrt{|\eta|})+1}\approx \sqrt{|\eta|}$ be the starting of the resonant interval. We then denote the whole resonant interval as $I_t(\eta)=[t(\eta),2|\eta|]=\bigcup_{m=1}^{E(\sqrt{|\eta|})}I_{m,\eta}$. Define $w(t,\eta)$ in the following
\begin{equation}
\begin{cases}
w(t,\eta)=1, & if\ t\leq t(\eta);\\
w(t,\eta)=w(t_{j,\eta},\eta)g_j(t-\frac{\eta}{j},\eta), \ \ \ & if\ t\in T_{j,\eta}, |j|\in [1,E(\sqrt{|\eta|})], j\eta>0;\\
w(t,\eta)=w(2|\eta|,\eta),\ &if \ t\geq2|\eta|.
\end{cases}
\end{equation}
Here, $g_m$ is defined by the following model
\begin{equation}
\begin{split}
\partial_{\tau}g_m=(\nu^{1/3}t_{m,\eta})^{-(1+\beta)}\frac{\nu^{1/3}\eta/m^2}{1+\tau^2}g_m,  \ \ \ g_m(-D^-_{m,\eta})=1,
\end{split}
\end{equation}
where $D^-_{m,\eta}=\frac{\eta}{(2m+1)m}=\frac{\eta}{m}-t_{m,\eta}$. By definition, $w(t,\eta)\approx1$ and when $|\xi-\eta|\leq |\eta|/10$, then it holds
\begin{equation}\label{lem6.1}
|w(t,\eta)-w(t,\xi)|\lesssim
\frac{|\xi-\eta|}{\la \eta\ra}\times \begin{cases}
\nu^{-1/3},\ \ \ \ \ t\lesssim \nu^{-1/3},\\
\nu^{\beta/3}t^{1-\beta},\ \ \ \ t\gtrsim \nu^{-1/3}.
\end{cases}
\end{equation}
See Lemma 6.1 in \cite{MZ19} for a proof. Define $\rho(x)$ to be a bounded smooth function such that $\rho(x)=0$ for $x\leq 1/20$ and $\rho(x)=1$ for $x\geq1/10$ and $\int_{1/20}^{1/10}\rho(x)dx=1/20$. Define $w_k(t,\eta)=w(t,\varrho(k,\eta))$ for $\varrho(k,\eta)=\eta$ when $k=0$ and $\varrho(k,\eta)=\frac{k}{20}+\int_0^{\eta}\rho(x/k)dx$ when $k\neq 0$. For such $\varrho(k,\eta)$, we have $\varrho(k,\eta)\approx \langle k,\eta\rangle$ and $|\varrho(k,\eta)-\varrho(l,\xi)|\lesssim |k-l,\xi-\eta|$ for $|k-l,\xi-\eta|\leq|l,\xi|/100$. 

\begin{lemma}[\cite{BM15}]
Let $\xi,\eta$ be such that there exists some $\alpha\geq1$ with $\alpha^{-1}|\xi|\leq |\eta|\leq \alpha|\xi|$ and let $k,n$ be such that $t\in I_{k,\eta}\cap I_{n,\xi}$, then $k\lesssim n$ and moreover at least one of the following holds: (i) $k=n$; (ii) $|t-\frac{\eta}{k}|\geq\frac{1}{10\alpha}\frac{\eta}{k^2}$ and $|t-\frac{\xi}{n}|\geq\frac{1}{10\alpha}\frac{\xi}{n^2}$; (iii) $|\eta-\xi|\gtrsim _{\alpha}\frac{|\eta|}{|n|}$.
\end{lemma}

%\begin{lemma}
%It holds that $w(t,\eta)\approx1$.
%\end{lemma}

As a consequence, with $w_k(t,\eta)$, we can define $A_{k}^{\sigma}(t,\eta)=\langle k,\eta\rangle^{\sigma}/w_k(t,\eta)$ and $A_{k}^{\sigma}(t,\eta)\approx\langle k,\eta\rangle^{\sigma}$.

%{\color{red}
%\begin{lemma}\label{lem4.4-1}
%Let $f\in H^N$ and $N>1$. Then for $\delta$ sufficiently small, there holds
%$$\|\sqrt{\frac{\partial_tw_k(t,\eta)} {w_k(t,\eta)}}A^{\sigma}\Delta_L\Delta_t^{-1} \Omega_{\neq}\|_{L^2L^2}\lesssim \|\sqrt{\frac{\partial_tw_k(t,\eta)} {w_k(t,\eta)}}A^{\sigma}\Omega_{\neq}\|_{L^2L^2}+(\cdots??).$$
%\end{lemma}
%\begin{proof}
%Note $\Delta_L=\Delta_t-G\partial_{YY}^L-(\partial_{yy}\psi+\bar U'')\partial_{Y}^L$. We have
%\begin{equation}
%\begin{split}
%&\|\sqrt{\frac{\partial_tw_k(t,\eta)} {w_k(t,\eta)}}  A^{\sigma}\Delta_L\Delta_t^{-1} \Omega_{\neq}\|_{L^2L^2}\lesssim \|\sqrt{\frac{\partial_tw_k(t,\eta)} {w_k(t,\eta)}}A^{\sigma} \Omega_{\neq}\|_{L^2L^2}\\ & \ \ \ \ \ \ +\|\sqrt{\frac{\partial_tw_k(t,\eta)} {w_k(t,\eta)}}A^{\sigma}G\partial_{YY}^L\Delta_t^{-1} \Omega_{\neq}\|_{L^2L^2} +\|\sqrt{\frac{\partial_tw_k(t,\eta)} {w_k(t,\eta)}}A^{\sigma}(\partial_{yy}\psi+\bar U'')\partial_{Y}^L\Delta_t^{-1} \Omega_{\neq}\|_{L^2L^2},
%\end{split}
%\end{equation}
%where $G=(\bar U'+\partial_y\psi)^2 -(\bar U')^2 =h(h-2\bar U')$. The second term can be estimated as
%\begin{equation}
%\begin{split}
%\|\sqrt{\frac{\partial_tw_k(t,\eta)} {w_k(t,\eta)}}A^{\sigma}G\partial_{YY}^L\Delta_t^{-1} \Omega_{\neq}\|_{L^2L^2} \lesssim
%\end{split}
%\end{equation}
%\end{proof}
%}

\subsection{Shear profile $\bar U(t,y)$}
With general initial data $(U(y),0)^{\top}$, we consider $\bar U(t,y)$, the solution of the 2D Navier-Stokes equation with initial data $(U(y),0)^{\top}$. Assume that the initial data is sufficiently close to Couette flow, in the sense that
$$\|\bar U'-1\|_{H^s} +\|\bar U''\|_{H^s}= \delta\leq \epsilon\nu^{5/4},$$
for some sufficiently small $\delta$ (depending on $\nu$), a uniform small constant $\epsilon$ (independent of $\nu$) and some large integer $s$.
From standard estimates on the heat equation, it holds
\begin{equation}\label{e25}
\begin{split}
\sup_{t>0}\|\bar U'(t,\cdot)-1\|_{H^s} \leq & \|\bar U'-1\|_{H^s},\\
\sup_{t>0}\|\bar U''(t,\cdot)-1\|_{H^s} \leq & \|\bar U''\|_{H^s},\\
and \ \ \ \ \ \|\bar U''\|_{L^2_tH^s_y}\leq & \delta\nu^{-1/2}.
\end{split}
\end{equation}

Some composition results in fractional Sobolev spaces are listed for readers' convenience. Details can be found in \cite{BVW18},  say.

\begin{lemma}
Let $s'>2$, $s'\geq s\geq0$, $f\in H^s(\Bbb R^n)$ and $g\in H^{s'}(\Bbb R^n)$, $n=1$, $2$,  such that $\|g\|_{H^{s'}}\leq \delta$. Then, there holds
$$\|f\circ (I+g)\|_{H^s}\leq C_{s,s'}(\delta)\|f\|_{H^s},$$
where $C_{s,s'}\to1$ as $\delta\to0$.
\end{lemma}

\begin{lemma}
For $\delta$ sufficiently small, there holds for $\sigma>2$ such that
$$\|\partial_Y((\bar  U'(t,y))^2-1)\|_{L^2_tH^{\sigma}_y} % =\|2b\|_{L^2_tH^{\sigma}_y} 
=\|\bar U''\|_{L^2_tH^{\sigma}_y}\lesssim \delta\nu^{-1/2},$$
$$%\|a-1\|_{H^{\sigma}}\leq 2
\|\bar U'-1\|_{H^{\sigma}}\lesssim \delta,\ \ \ and \ \ \ \  %\|b\|_{H^{\sigma}} \leq 2 
\|\bar U''\|_{H^{\sigma}} \lesssim \delta.$$
\end{lemma}

\subsection{Littlewood-Paley decomposition and paraproducts}
We define the following paraproduct decomposition, introduced by Bony \cite{BCD2011,Bony81}.  Let $f(x)$ be in the Schwartz space and define the Fourier transform $\hat f(\xi)$ as
$$\mathcal F(f)(\xi)=\hat f(\xi)=\frac{1}{\sqrt{2\pi}}\int_{\bf R}e^{-x\cdot\xi}f(x)dx,$$
and the Fourier inverse transform
$$\mathcal F^{-1}(\hat f)(x)=(\hat f)^{\vee}(x)=f(x)=\frac{1}{\sqrt{2\pi}}\int_{\bf R}e^{x\cdot\xi}\hat f(\xi)d\xi.$$
Let $\psi\in C_0^{\infty}(\bf R,\bf R)$ be such that $\psi(\xi)=1$ for $|\xi|\leq 1/2$ and $\psi(\xi)=0$ for $|\xi|\geq 3/4$ and define $\chi(\xi)=\psi(\xi/2)-\psi(\xi)$ supported in the range $\xi\in(1/2,3/2)$. Then we have the partition of unity
$$1=\psi(\xi)+\sum_{M\in 2^{\Bbb N}}\chi_M(\xi),$$
where $\Bbb N=\{0,1,2,3,\cdots,j,\cdots\}$, $M=\{1,2,4,8,\cdots,2^j,\cdots\}$ and $\chi_M(\xi)=\chi(M^{-1}\xi)$. For $f\in L^2(\bf R)$, define
\begin{equation}
\begin{split}
f_M= & (\chi_M(\xi)\hat f(\xi))^{\vee}, \ \ \    f_{\frac12}=  (\psi(\xi)\hat f(\xi))^{\vee} \ \ \ and \ \ \   f_{<M}=   f_{\frac12}+\sum_{K\in 2^{\Bbb N},K<M}f_K
\end{split}
\end{equation}
and hence the decomposition
$$f=f_{\frac12}+\sum_{K\in 2^{\Bbb N}}f_K.$$

Given suitable functions $f$ and $g$, we define the paraproduct decomposition as
\begin{equation}
\begin{split}
fg=&T_fg+T_gf+\mathcal R(f,g) 
= \sum_{N\geq8}f_{<N/8}g_N +\sum_{N\geq8}f_Ng_{<N/8} +\sum_{N\in\Bbb D}\sum_{N/8\leq N'\leq 8N}g_{N'}f_N,
\end{split}
\end{equation}
where all the sums are understood to run over $\Bbb D$.

\subsection{Decay estimate for $e^{\nu t\Delta }$}
We consider the following initial value problem for the heat equation
$$\partial_tf=\nu\Delta f,\ \ \ f|_{t=0}=f_0.$$
\begin{prop}\label{prop5.1}
For any $1\leq l\leq r\leq \infty$, we have for $t>0$,
\begin{equation}
\begin{split}
\|e^{\nu t\Delta}f_0\|_{L^r}\leq & C\frac{1}{(\nu t)^{\frac{n}{2}(\frac1l-\frac1r)}}\|f_0\|_{L^l}\ \ \ \ and\\
\|\partial^{\alpha}e^{\nu t\Delta}f_0\|_{L^r}\leq & C\frac{1}{(\nu t)^{\frac{|\alpha|}{2}+\frac{n}{2}(\frac1l-\frac1r)}}\|f_0\|_{L^l},\ \ \ |\alpha|\geq1.
\end{split}
\end{equation}
\end{prop}
In particular, since $\bar U(t,y)-y$ satisfies the one dimensional heat equation
$$\partial_t(\bar U-y)=\nu \partial_{yy}(\bar U-y),\ \ \ \ (\bar U-y)|_{t=0}=U(y)-y,$$
then we have
\begin{equation}\label{L2small}
\begin{split}
\sup_{t\in[0,\infty]}\|\bar U-y\|_{L^2}\leq & \|U(y)-y\|_{L^2},\ \ \ N\geq0,\\
\|\bar U'-1\|_{L^2}\leq & C\min\{1,{(\nu t)^{-3/4}}\}\|U(y)-y\|_{L^2},\ \ \ N\geq0,\\
\|\bar U''\|_{H^N}\leq & C\min\{1,{(\nu t)^{-5/4}}\}\|U(y)-y\|_{L^2},\ \ \ N\geq0,
\end{split}
\end{equation}
and
\begin{equation}\label{L1small}
\begin{split}
\|\bar U'-1\|_{L^2}\leq & C{(\nu t)^{-3/4}}\|U(y)-y\|_{L^1},\ \ \ N\geq0,\\
\|\bar U''\|_{H^N}\leq & C{(\nu t)^{-5/4}}\|U(y)-y\|_{L^1},\ \ \ N\geq0.
\end{split}
\end{equation}
After integration, we have
\begin{equation}\label{integration}
\int_1^t\|\bar U''\|_{H^N}d\tau\leq C\nu^{-5/4}\|U(y)-y\|_{L^1},\ \ \ \ N\geq0,
\end{equation}
which implies
\begin{equation}\label{integration1}
\int_1^t\|\bar U''\|_{H^N}d\tau\leq C\epsilon,\ \ \ \forall t\geq1,
\end{equation}
upon assuming $\|U(y)-y\|_{L^1}\leq \epsilon\nu^{5/4}$ for some universal $0<\epsilon<1$.

\bigskip
\noindent {\bf Acknowledgments.}
D. Bian is supported by NSFC under the contract 11871005. X. Pu is supported by NSFC under the contract 11871172 and Natural Science Foundation of Guangdong Province of China under 2019A1515012000.

\small

\begin{center}

\end{center}
%\end{CJK*}

\begin{thebibliography}{99}
\addcontentsline{toc}{section}{References} {\small

\bibitem{BCD2011} H. Bahouri, J.-Y. Chemin and R. Danchin, Fourier Analysis and Nonlinear Partial Differential Equations, Grundlehren der Mathematischen Wissenschaften Fundamental Principles of Mathematical Sciences, vol. 343, Springer, Heidelberg, 2011.

\bibitem{BGM15a} J. Bedrossian, P. Germain, and N. Masmoudi, Dynamics near the subcritical transition of the 3D Couette flow I: Below threshold, {\it Mem. of the AMS}, 266(1294), (2020)v+158.

\bibitem{BGM15b} J. Bedrossian, P. Germain, and N. Masmoudi, Dynamics near the subcritical transition of the 3D Couette flow II: above threshold, arXiv:1506.03721.

\bibitem{BGM17} J. Bedrossian, P. Germain, and N. Masmoudi, On the stability threshold for the 3D Couette flow in Sobolev regularity, \emph{Ann. Math.}, 185, (2017)541-608.

\bibitem{BGM19} J. Bedrossian, P. Germain, and N. Masmoudi, Stability of the Couette flow at high Reynolds number in two dimensions and three dimensions, \emph{Bull. Amer. Math. Soc.}, 56(3), (2019)373-414.

\bibitem{BM15}  J. Bedrossian and N. Masmoudi, Inviscid damping and the asymptotic stability of planar shear flows in the 2D Euler equations, \emph{Publ. Math. l'IHES.}, 122(1), (2015)195-300.

\bibitem{BMV16} J. Bedrossian, N. Masmoudi and V. Vicol, Enhanced dissipation and inviscid damping in the inviscid limit of the Navier-Stokes equations near the 2D Couette flow, \emph{Arch. Ration. Mech. Anal.}, 216(3), (2016)1087-1159.

\bibitem{BVW18} J. Bedrossian, V. Vicol and F. Wang, The Sobolev stability threshold for 2D shear flows near Couette, \emph{J. Nonl. Sci.}, 28, (2018)2051-2075.

\bibitem{BP21} D. Bian and X. Pu, Stability threshold for 2D shear flows of the Boussinesq system near Couette, arXiv:2012.02386, 2021.

\bibitem{BZD2005} R. Bianchini, M. Coti Zelati and M. Dolce, Linear inviscid damping for shear flows near Couette in the 2D stably stratified regime, arXiv:2005.09058v1, 2020.

\bibitem{Bony81} J. Bony, Calcul symbolique et propagation des singularit\'{e}s pour les \'{e}quations aux d\'{e}riv\'{e}es partielles non-lin\'{a}ires, Ann. Sci. \'{E}c. Norm. Super., 14 (1981), 209-246.

\bibitem{CD81} J.R. Cannon and E. Di Benedetto, The initial problem for the Boussinesq equations with data in $L^p$, \emph{Lecture Notes in Mathematics}, Vol. {771}, Springer: Berlin, 1980.

\bibitem{Chae06} D. Chae, Global regularity for the 2D Boussinesq equations with partial viscosity terms, \emph{Adv. Math.}, 203(2), (2006)497-513.

\bibitem{CN97} D. Chae and H.-S. Nam, Local existence and blow-up criterion for the Boussinesq equations, \emph{Proc. R. Soc. Edinb. Sect. A}, {127}, (1997)935-946.

\bibitem{DM18} Y. Deng and N. Masmoudi, Long time instability of Couette flow in low Gevrey spaces, arXiv:1803.01246v1, 2018.

\bibitem{DWZ2004} W. Deng, J. Wu and P. Zhang, Stability of Couette flow for 2D Boussinesq system with vertical dissipation, arXiv:2004.09292v1, 2020.

%\bibitem{DM18} Y. Deng and N. Masmoudi, Long time instability of the Couette flow in low Gevrey spaces, arXiv:1803.01246v1, 2018.

%\bibitem{DZ19} Y. Deng and C. Zillinger, Echo chains as a linear mechanism: Norm inflation, modified exponents and asymptotics, arXiv:1910.12914v1, 2019.

\bibitem{DWZZ18} C.R. Doering, J. Wu, K. Zhao and X. Zheng, Long time behavior of the two-dimensional Boussinesq equations without buoyancy diffusion, \emph{Physica D: Nonlinear Phenomena}, 376, (2018)144-159.

%\bibitem{Drazin81} P. Drazin and W. Rein, \emph{Hydrodynamic stability}, Cambridge University Press, 1981.

\bibitem{Gold} S. Goldstein, On the stability of superposed streams of fluids of different densities, \emph{Proc. R. Soc. Lond. A}, 132(820), (1931)524-548.

\bibitem{G-G-T16} E. Grenier, Y. Guo, T. T. Nguyen, Spectral instability of general symmetric shear flows in a two-dimensional channel,  \emph{Adv. Math.}, 292, (2016) 52-110.

%\bibitem{HLL08}  T.Y. Hou, Z. Lei and C. Li, Global Regularity of the 3D Axi-Symmetric Navier-Stokes Equations with Anisotropic Data, \emph{Commun. Partial Differential Equations}, {\bf33}, (2008)1622-1637.

\bibitem{HL05} T.Y. Hou and C. Li, Global well-posedness of the viscous Boussinesq equations, \emph{Discret. Cont. Dyn. Sys.}, {12} (2005), 1-12.

\bibitem{Jia20} A.D. Ionescu and H. Jia, Inviscid damping near the Couette flow in a channel, \emph{Comm. Math. Phys.}, 374(3), (2020)2015-2096.

\bibitem{Jia20SIAM} H. Jia, Linear inviscid damping near monotone shear flows, \emph{SIAM J. Math. Anal.}, 52(1), (2020)623-652.

\bibitem{Kelvin1887}  L. Kelvin, Stability of fluid motion-rectilinear motion of viscous fluid between two parallel plates, Phil. Mag., 24, (1887)188-196.

\bibitem{LMZ22} H. Li, N. Masmoudi and W. Zhao, New energy method in the study of the instability near Couette flow, arXiv: 2203.10894v1, 2022.

\bibitem{LZ11} Z. Lin and C. Zeng, Inviscid dynamical structures near Couette flow, \emph{Arch. Ration. Mech. Anal.}, 200, (2011)1075-1097.


\bibitem{MSZ20} N. Masmoudi, B. Said-Houari and W. Zhao, Stability of Couette flow for 2D Boussinesq system without thermal diffusivity, arXiv:2010.01612v1, 2020.

\bibitem{MZ19} N. Masmoudi and W. Zhao, Stability threshold of the 2D Couette flow in Sobolev spaces, arXiv:1908.11042v1. 2019.

\bibitem{Orr} W. Orr, The stability or instability of steady motions of a perfect liquid and of a viscous liquid, Part I: a perfect liquid, \emph{Proc. R. Ir. Acad. Sect. A Math. Phys. Sci.}, 27, (1907)9-68.

\bibitem{Rayleigh1880} L. Rayleigh, On the stability or instaiblity of certain fluid motions, Proc. London Math. Soc., 9, (1880)57-70.

\bibitem{Reddy} S. Reddy, P. Schmid, J. Baggett and D. Henningson, On stability of streamwise streaks and transition thresholds in plane channel flows, \emph{J. Fluid Mech.}, 365, (1998)269-303.

\bibitem{Rom73} V.A. Romanov, Stability of plane-parallel Couette flow, \emph{Funk. Anal. i. Prilozen}, 7(1973)62-73.

\bibitem{Synge} J.L. Synge, The stability of heterogeneous liquids, \emph{Trans. Royal Soc. Canada}, 1993.

\bibitem{TW19} L. Tao and J. Wu, The 2d Boussinesq equations with vertical dissipation and linear stability of shear flows, \emph{J. Differential Equations}, 267(3)(2019)1731-1747.

\bibitem{Tay} G.I. Taylor, Effect of variation in density on the stability of superposed streams of fluid, \emph{Proc. Royal Society London. A.}, 132(820):499-523, 1931.

\bibitem{Vil09} C. Villani, Hypocoercivity, \emph{Mem. Amer. Math. Soc.}, 202(950), (2009)iv+141.

\bibitem{WZ18} D. Wei and Z. Zhang, Transition threshold for the 3D Couette flow in Sobolev space, arXiv:1803.01359v1, 2018

\bibitem{WZZ18} D. Wei, Z. Zhang and W. Zhao,  Linear inviscid damping for a class of monotone shear flow in Sobolev spaces, \emph{Comm. Pure Appl. Math.}, 71(4), (2018)617-687.

\bibitem{WZZ20} D. Wei, Z. Zhang and W. Zhao, Linear inviscid damping and enhanced dissipation for the Kolmogrov flow, \emph{Adv. Math.}, 362, (2020)106963.

\bibitem{YL18} J. Yang and Z. Lin, Linear inviscid damping for Couette flow in stratified fluid, \emph{J. Math. Fluid Mech.}, 20(2), (2018)445-472.

\bibitem{Zill17} C. Zillinger, Linear inviscid damping for monotone shear flows, \emph{Trans. Amer. Math. Soc.}, 369(12), (2017)8799-8855.

\bibitem{Zill20} C. Zillinger, On enhanced dissipation for the Boussinesq equations, \emph{J. Differential Equations}, 282, (2021)407-445.

\bibitem{Zill2011} C. Zillinger, On the Boussinesq equation with non-monotone temperature profiles, arXiv: 2011.02316v1, 2020.

}
\end{thebibliography}
\end{document}